\documentclass[12pt,a4paper]{amsart}
\usepackage{hyphenat}
\usepackage{fix-cm}
\usepackage[a4paper,left=1.8cm,right=1.8cm,top=2.5cm,bottom=2.5cm]{geometry}

\usepackage{microtype}
\usepackage[table]{xcolor}
\usepackage{amsmath}
\usepackage{amstext}
\usepackage{amsfonts}
\usepackage{amssymb}
\usepackage{soul}
\usepackage{amsbsy}
\usepackage{latexsym}
\usepackage[T1]{fontenc}
\usepackage{lmodern}
\usepackage[all]{xy}
\usepackage{tikz-cd}
\usepackage{hhline}

\usepackage{enumitem}
\usepackage{float}
\setlist{nosep}
\setlist[enumerate]{label=(\roman*)}
 
\usepackage{etoolbox}
\AtBeginEnvironment{thebibliography}{%
  \setlength{\itemsep}{0pt}%
  \setlength{\parskip}{0pt}%
  \setlength{\topsep}{0pt}%
  \setlength{\partopsep}{0pt}%
}
 
\newlist{defenum}{enumerate}{3}
\setlist[defenum]{label=(\alph*),leftmargin=*,align=left}
\newlist{defsubenum}{enumerate}{1}
\setlist[defsubenum]{label=(\roman*),leftmargin=*,align=left}
\usepackage{dsfont}
\usepackage{makecell}
\usepackage{graphicx}
\usepackage{caption}
\usepackage{colortbl}
\usepackage{booktabs} 
\usepackage{thmtools}
\allowdisplaybreaks

\usepackage[colorlinks=true,pagebackref,hypertexnames=false,linkcolor=red,citecolor=blue,filecolor=blue,urlcolor=blue]{hyperref}
 
\usepackage[nameinlink,noabbrev]{cleveref}
\usepackage{autonum}
\usepackage[alphabetic,backrefs]{amsrefs}
 
\makeatletter
\renewcommand\normalsize{%
    \@setfontsize\normalsize{11.7}{14pt plus .3pt minus .3pt}%
    \abovedisplayskip 10\p@ \@plus4\p@ \@minus4\p@
    \abovedisplayshortskip 6\p@ \@plus2\p@
    \belowdisplayshortskip 6\p@ \@plus2\p@
    \belowdisplayskip \abovedisplayskip}
\renewcommand\small{%
    \@setfontsize\small{10}{12\p@ plus .2\p@ minus .2\p@}%
    \abovedisplayskip 8.5\p@ \@plus4\p@ \@minus1\p@
    \belowdisplayskip \abovedisplayskip
    \abovedisplayshortskip \abovedisplayskip
    \belowdisplayshortskip \abovedisplayskip}
\renewcommand\footnotesize{%
    \@setfontsize\footnotesize{8.5}{9.25\p@ plus .1pt minus .1pt}
    \abovedisplayskip 6\p@ \@plus4\p@ \@minus1\p@
    \belowdisplayskip \abovedisplayskip
    \abovedisplayshortskip \abovedisplayskip
    \belowdisplayshortskip \abovedisplayskip}
\ifdefined\pdfpagewidth
\else
\fi
\calclayout
\makeatother
 
  {}              
 
{\color{blue}} 
{}              

\newcounter{num}[section] %
\renewcommand{\thenum}{\thesection.\arabic{num}} %

\newtheorem{Mtheorem}{Theorem}

\newenvironment{theo}
{\refstepcounter{num}%
 \bigskip\noindent{\bf Theorem~\arabic{section}.\arabic{num}.}\quad\itshape}
{\par\bigskip}

\newenvironment{prop}
{\refstepcounter{num}%
 \bigskip\noindent{\bf Proposition~\arabic{section}.\arabic{num}.}\quad\itshape}
{\par\bigskip}

\newenvironment{cor}
{\refstepcounter{num}%
 \bigskip\noindent{\bf Corollary~\arabic{section}.\arabic{num}.}\quad\itshape}
{\par\bigskip}

\newenvironment{lemma}
{\refstepcounter{num}%
 \bigskip\noindent{\bf Lemma~\thenum.}\quad\itshape}
{\par\bigskip}

\newenvironment{example}
{\refstepcounter{num}%
 \bigskip\noindent{\bf Example~\arabic{section}.\arabic{num}.}\quad}
{\par\bigskip} 

\newenvironment{remark}
{\refstepcounter{num}%
 \bigskip\noindent{\bf Remark~\arabic{section}.\arabic{num}.}\quad}
{\par\bigskip}

\newenvironment{defin}
{\refstepcounter{num}%
 \bigskip\noindent{\bf Definition~\arabic{section}.\arabic{num}.}\quad}
{\par\bigskip}

\newcommand{\tr}{\mathop{\rm tr}}

\renewcommand{\d}{\mathrm{d}}

\newcommand{\ad}{\mathrm{ad}}

\newcommand\restr[2]{{
  \left.\kern-\nulldelimiterspace 
  #1 
  \littletaller 
  \right|_{#2} 
}}

\newcommand{\littletaller}{\mathchoice{\vphantom{\big|}}{}{}{}}

\newcommand{\mylabel}[1]{}

\begin{document}

 \title{On the generalized geometry of almost abelian solvmanifolds}

\author[Andrade]{Mauro de Andrade Pinto}
\address{Instituto de Matemática, Estatística e Computação Científica (IMECC) da Universidade Estadual de Campinas (Unicamp), Cidade Universitária, Campinas - SP, 13083-856, Brazil.}
\email{mauroap1415@gmail.com}

\author[Camponês do Brasil]{Letícia Camponês do Brasil Maia}
\address{Instituto de Matemática, Estatística e Computação Científica (IMECC) da Universidade Estadual de Campinas (Unicamp), Cidade Universitária, Campinas - SP, 13083-856, Brazil.}
\email{leticiacampones@gmail.com}

\author[Cavenaghi]{Leonardo F. Cavenaghi}
\address{Institute of Mathematics and Informatics, Bulgarian Academy of Sciences, Sofia, Bulgaria}
\email{leonardofcavenaghi@gmail.com}

\author[Martins]{Pedro Antonio Muniz Martins}
\address{Instituto de Matemática, Estatística e Computação Científica (IMECC) da Universidade Estadual de Campinas (Unicamp), Cidade Universitária, Campinas - SP, 13083-856, Brazil.} 
\email{pedroa.muniz9@gmail.com}

\begin{abstract}
We study left-invariant generalized complex structures on almost abelian Lie groups $G_A$ with Lie algebra $\mathfrak g_A=\mathbb Re_0\ltimes_A\mathfrak h$, where $\mathfrak h$ is an abelian ideal of codimension one, and on their compact quotients. First, when $A$ is diagonalizable over $\mathbb R$, we characterize all admissible types by pairings of its eigenvalues and characterize the structures admitting a closed invariant pure-spinor generator. For general $A$, we obtain Jordan-theoretic type bounds and a construction using a complex quotient and a symplectic ideal. Finally, in dimension six, we establish nonexistence results and give explicit intermediate-type constructions. 

\end{abstract}



\maketitle

\section{Introduction}
Generalized complex geometry, in the sense of Hitchin, Gualtieri and Cavalcanti
\cites{Hitchin2003, Gualtieri2011, Cavalcanti2004}, provides a common framework for symplectic and
complex geometry. Let $M$ be a smooth manifold of real dimension $2n$ and endow $T_M\oplus T^*_M$ with
the split-signature pairing
$$\langle X + \xi, Y + \eta \rangle := \tfrac{1}{2}\left( \xi(Y) + \eta(X) \right).$$
A \emph{generalized complex structure} on $M$ is an endomorphism $\mathcal J$ of $T_M\oplus T^*_M$
with $\mathcal J^2=-\mathrm{Id}$, orthogonal for this pairing, whose $\mathbf i$-eigenbundle is
involutive for the Courant bracket. Every symplectic structure on $M$ is an example of \emph{type
$0$}, and every complex structure on $M$ is an example of \emph{type $n$}. As we recall below, the
type is a pointwise invariant and may jump from point to point; for left-invariant structures on a
Lie group, however, it is constant. Let us be precise.

\medskip

Consider the complexification
$$\left(T_M\oplus T_M^\ast\right)\otimes_{\mathbb R}\mathbb C=L\oplus \bar L,\qquad
L:=\ker\left(\mathcal J-\mathbf{i}\,\mathrm{Id}\right).$$
We ask generalized complex structures to be, by definition, \emph{integrable}, which means
$[[C^\infty(L), C^\infty(L)]]\subset C^\infty(L)$, after the obvious extension of the Courant bracket
$[[\cdot,\cdot]]$ to the complexification. The condition $\mathcal J^2=-\mathrm{Id}$ gives
$L \cap \overline{L} = \{0\}$, while orthogonality forces $L$ to be isotropic. By rank count, $L$ is
in fact maximal isotropic. As such, $L$ is the \emph{Clifford annihilator}
$$L = \{v \in (T_M \oplus T^{*}_M) \otimes_{\mathbb R} \mathbb{C} \mid v \cdot K_L = 0 \}$$
of a unique line sub-bundle $K_L \subset \bigwedge^\bullet T^*_M \otimes_{\mathbb R} \mathbb{C}$, where
the dot denotes the Clifford action $(X+\xi)\cdot\alpha=\iota_X\alpha+\xi\wedge\alpha$. A line whose
annihilator is maximal isotropic is, by definition, a \emph{pure} spinor line, and we call $K_L$
\emph{the canonical line bundle of $\mathcal J$}. It admits, around each point of $M$, a smooth local
generator $\rho \in C^{\infty}(U, \bigwedge^\bullet T^*_M \otimes_{\mathbb R} \mathbb{C})$, and purity of
$\rho$ means that at each point $x \in U$ one can write
$$\rho_x = \exp(B+\mathbf{i}\omega)\wedge\Omega,$$ 
where $B,~\omega$ are real two-forms and, for $k>0$, $\Omega=\theta_1\wedge\cdots\wedge\theta_k$ with independent
complex covectors. At type zero, $\Omega$ is a nonzero complex
scalar; after normalizing the generator, it can be taken to be $1$. The integer $k$ is the
\emph{type} of $\mathcal J$ at $x$, and it is precisely this integer that need not be constant. When
$\rho$ can moreover be chosen to be a global nowhere-vanishing and \emph{closed} generator, the structure is called \emph{generalized Calabi--Yau}. We separately say that an invariant generalized complex structure \emph{admits a closed invariant generator} when a closed left-invariant pure spinor generates its canonical line. Throughout this paper, an \emph{invariant generalized Calabi--Yau
structure} means this stronger condition. Invariance of the generalized
complex structure alone does not require its closed generator to be
invariant. Its type, however, is constant and determined by the Lie algebra.

\medskip

This paper concerns left-invariant generalized complex structures on almost abelian Lie groups and,
consequently, on the associated solvmanifolds. Recall that a solvmanifold is a compact quotient
$\Gamma\backslash G$, where $G$ is a simply connected solvable Lie group and $\Gamma<G$ is a lattice,
that is, a discrete cocompact subgroup. Since the elements of $\Gamma\backslash G$ are the cosets
$\Gamma g$, every left-invariant tensor on $G$ is in particular $\Gamma$-invariant and therefore
descends to the quotient. An \emph{almost abelian solvmanifold} is one of the form $\Gamma\backslash G_A$, with $G_A$ a simply connected almost abelian Lie group. The invariant generalized complex structures on $\Gamma\backslash G_A$ thus correspond to the left-invariant ones on $G_A$, which in turn are determined at the level of the Lie
algebra $\mathfrak g_A=\mathbb Re_0\ltimes_A\mathfrak h$, where $\mathfrak h$ is an abelian ideal of
codimension one and the whole bracket is encoded in the matrix $A$ of $\ad_{e_0}|_{\mathfrak h}$. We
work throughout at the level of $\mathfrak g_A$, passing to the quotients when the existence of a
lattice matters; recall that this forces $\tr A=0$, as a Lie group admitting a lattice is unimodular.
A central question organizes the work: which types of generalized complex structures occur on $\mathfrak{g}_A$. This is answered completely when $A$ is diagonalizable over $\mathbb R$. In general, we give an algebraic criterion for the existence of such structures and we constrain the admissible types by Jordan-theoretic bounds. Finally, in dimension $6$ we give nonexistence results and explicit constructions.



\medskip



In the diagonalizable case, by Proposition \ref{Conjugate}, we may assume $A$ is diagonal. The classification is then expressed entirely through the spectrum of $A$ and governed by two patterns: pairs of \emph{identical} eigenvalues and pairs of \emph{opposite} eigenvalues. Every admissible spectrum in this setting is a combination of the two, together with one remaining eigenvalue.

\begin{Mtheorem}\label{thmB}
Let $G_A$ be a $2n$-dimensional almost abelian Lie group with $A$ diagonal. Then $G_A$ admits a
left-invariant generalized complex structure of type $k$ if and only if the eigenvalues of $A$, listed
with multiplicity, can be reordered as
$$\{\underbrace{\lambda_1,\lambda_1,\ldots,\lambda_{k},\lambda_{k}}_{k\text{ identical pairs}},\ \mu,\
\underbrace{\zeta_1,-\zeta_1,\ldots,\zeta_{n-k-1},-\zeta_{n-k-1}}_{n-k-1\text{ opposite pairs}}\},
\qquad k\le n-1,$$
or as
$$\{\underbrace{\lambda_1,\lambda_1,\ldots,\lambda_{k-1},\lambda_{k-1}}_{k-1\text{ identical pairs}},\
\mu,\ \underbrace{\zeta_1,-\zeta_1,\ldots,\zeta_{n-k},-\zeta_{n-k}}_{n-k\text{ opposite pairs}}\},
\qquad k\ge 1 .$$
\end{Mtheorem}

\medskip

The extremal cases $k=0$ and $k=n$ of Theorem \ref{thmB} are the classification of left-invariant
symplectic and complex structures on almost abelian Lie groups obtained in \cite{Arroyo2025}, which is
also our starting point in the proof; the case $k=0$ recovers as well the classification of left-invariant symplectic structures of \cite{LuisPedro2022}. Two consequences are worth emphasizing. First, the two spectral
forms coincide for consecutive types, so the types come in adjacent pairs: a group that admits a
structure of type $k$ admits one of type $k+1$ or of type $k-1$. Second, and in contrast with the
extremal cases, there exist almost abelian Lie groups carrying only structures of intermediate type,
hence neither complex nor symplectic. Requiring in addition that the pure spinor be closed refines
Theorem \ref{thmB} by a single linear condition on the eigenvalues involved.

\begin{Mtheorem}\label{thmC}
In the situation of Theorem \ref{thmB}, there exists a structure of type $k$ admitting a closed invariant pure-spinor generator if and only if the eigenvalues admit one of the stated reorderings with $\sum_{i=1}^k\lambda_i=0$ in the first form, or $\mu+\sum_{i=1}^{k-1}\lambda_i=0$ in the second.
\end{Mtheorem}

We stress that Theorems \ref{thmB} and \ref{thmC} are stated for the Lie group $G_A$, without restrictions
 on $\tr A$; imposing the existence of a lattice adds the condition $\tr A=0$, which
determines the remaining eigenvalue $\mu$ in terms of the others. In Example \ref{ex:dim10-lattice}, we provide a generalized complex structure on a solvmanifold after showing how a lattice can be made compatible with data defining a left-invariant generalized complex structure.

\medskip

Based on the approach developed in~\cite{ACK}, we also give an algebraic criterion for the existence of a generalized complex structure of prescribed type.
\begin{Mtheorem}\label{Mtheo:Gen_com_struct_alg}
Let $G_A$ be a $2n$-dimensional almost abelian Lie group with Lie algebra
$\mathfrak{g}_A = \mathbb{R}e_0 \ltimes_A \mathfrak{h}$. Suppose that $\mathfrak{h}$ decomposes into
$A$-invariant subspaces $\mathfrak{h} = \mathfrak{h}_c \oplus \mathfrak{h}_s$, with
$\dim \mathfrak{h}_c = 2k-1$ and $\dim \mathfrak{h}_s = 2(n-k)$, where $1\le k\le n$, and write $A_c=A|_{\mathfrak h_c}$
and $A_s=A|_{\mathfrak h_s}$. Assume that:
\begin{enumerate}
    \item the quotient Lie algebra $\mathfrak{g}_c := \mathfrak{g}_A/\mathfrak{h}_s \cong
    \mathbb{R}e_0 \ltimes_{A_c} \mathfrak{h}_c$ admits a complex structure;
    \item there exists a closed $2$-form $\omega \in \bigwedge^2 \mathfrak{g}_A^*$ restricting to a
    non-degenerate form $\sigma \in \bigwedge^2 \mathfrak{h}_s^*$ on the ideal $\mathfrak{h}_s$.
\end{enumerate}
Then $G_A$ admits a left-invariant generalized complex structure of type $k$.
\end{Mtheorem}

Condition (ii) is purely spectral: by Proposition \ref{prop:jordan_equivalence} it holds exactly when
the Jordan blocks of $A_s$ pair up in a way we make explicit, so that Theorem
\ref{Mtheo:Gen_com_struct_alg} becomes an algorithm for producing examples in any dimension, with no
diagonalizability assumption.

\medskip

While Theorem \ref{Mtheo:Gen_com_struct_alg} provides a general constructive method, the complete classification beyond the diagonalizable case remains open\footnote{R.~Arroyo, personal communication: the authors of \cite{Arroyo2025} are working in this direction.}, but the type is constrained by the Jordan structure of $A$, and the constraint is severe: a single Jordan block per eigenvalue forces the
type to be extremal or nearly so.

\begin{Mtheorem}\label{thmD}
Let $G_A$ be an almost abelian Lie group with $\mathrm{spec}(A)\subset\mathbb R$ and every eigenvalue
of geometric multiplicity $1$. Then every left-invariant generalized complex structure on $G_A$ has
type $k\le1$. More generally, for arbitrary Jordan type, if the generalized eigenspace of a real
eigenvalue $\lambda$ decomposes into $s_\lambda$ real Jordan blocks of sizes
$m_1^{(\lambda)}\ge\cdots\ge m_{s_\lambda}^{(\lambda)}$ with $m_\lambda=\sum_i m_i^{(\lambda)}$, then
$$k \leq \min_{\lambda} \left( \sum_{j=1}^{\lfloor s_\lambda/2 \rfloor} m_{2j}^{(\lambda)} +
\left\lfloor \frac{2n - 1 - m_\lambda}{2} \right\rfloor +1 \right),$$
the minimum being taken over the real eigenvalues of $A$.
\end{Mtheorem}

Theorem \ref{thmD} constrains the type without determining it. Four-dimensional solvable Lie groups can carry invariant structures of intermediate type. Theorem 4.7 of \cite{Barberis} instead states that the existence of an invariant generalized complex structure in dimension four is equivalent to the existence of an invariant complex or symplectic structure. Thus, dimension six is the first dimension in which an almost abelian Lie group can admit invariant generalized complex structures while admitting neither invariant complex nor invariant symplectic structures. We study the ten parameterized Jordan block patterns listed in Section \ref{sec:dimension-6}.

\begin{Mtheorem}\label{thmF}
Let $\mathfrak g_A=\mathbb Re_0\ltimes_A\mathbb R^5$ admit neither an invariant complex nor an invariant symplectic structure. The representatives in Table \ref{tab:forms}, under their assigned conditions, define invariant generalized complex structures of the indicated types, and their invariant generators are closed when the corresponding value of $\Lambda$ is zero. There is no invariant generalized complex structure of any type for
\[
A=J_5(\lambda)\quad(\lambda\ne0),\quad
A=J_4(\lambda)\oplus J_1(\mu)\quad(\lambda\ne0),\quad
A=J_3(\lambda)\oplus J_2(\mu)\quad(\lambda\ne\pm\mu).
\]
For diagonal $A$, Proposition \ref{prop:diag6} gives a complete criterion for existence of types $1$ and $2$.
\end{Mtheorem}

Three points concerning these results deserve remark. First, it is genuinely finer than what the
Jordan data alone predict: the bound of Theorem \ref{thmD} permits type $1$ for $A=J_5(\lambda)$, yet
no such structure exists unless $\lambda=0$, and the three configurations above carry no generalized
complex structure whatsoever, extremal or not. Second, the obstruction in each case is the interaction
between integrability and non-degeneracy rather than either alone: integrability forces a linear system
on the coefficients of $\omega$ whose only solution kills the coefficients that non-degeneracy requires
to be nonzero. Third, asking an invariant generator to be closed imposes an additional trace condition. Under the exclusion of extremal structures, the
diagonal case admits no closed invariant generator. Among the representatives verified below, a closed invariant generator occurs in the $J_2+1+1+1$ family described in Table \ref{tab:forms}.

\medskip

In this paper we do not address generalized K\"ahler structures; for almost abelian Lie groups this problem has already been studied by Fino and Paradiso~\cite{Fino2021}. However, the methods developed here could also be applied to this setting and may provide some new insights. From Gualtieri's work~\cite{Gualtieri2011} it is known that a generalized K\"ahler structure is equivalent to the existence of a bihermitian pair $(g,J_+)$, $(g,J_-)$ of SKT structures whose Bismut torsion $3$-forms are opposite. This reformulation translates the problem of studying generalized K\"ahler structures into that of understanding Hermitian connections with totally skew-symmetric torsion, a class of structures extensively studied in~\cite{Ivanov2013, Alexandrov2001, Friedrich2002}.

\medskip

  \vspace{1em }

    \subsection*{Acknowledgments}
    M. Andrade~P. is supported by FAPESP grant no. 2025/14730-8.

    L.~Camponês do Brasil M. is supported by FAPESP grants no. 2025/23958-2, 2026/09294-7. 
        
    L.~F.~Cavenaghi is supported by the Simons Foundation, grant SFI-MPS-T-Institutes-00007697, and the Ministry of Education and Science of the Republic of Bulgaria, grant DO1-239/10.12.2024.
    
    P.~A.~M.~Martins is supported by FAPESP grants no. 2024/07684-7, 2025/22312-1. 

   The authors are grateful to Lino Grama for suggesting the project and for his guidance and continued collaboration throughout its development.

    The authors thank Giovane Galindo, who kindly helped with fruitful insights during the conception of this manuscript.

The authors thank Beatrice Brienza, whose comments and questions led to Theorem \ref{Mtheo:Gen_com_struct_alg}.

    This work was conceived while P.~A.~M.~Martins, L.~Camponês do Brasil M., and  M. Andrade~P. were visiting the ICMS-Sofia under L.~Katzarkov's research group ``Theory of Atoms'', for which they thank the hospitality and leadership.
    \vspace{1em}
\bigskip


\section{Generalized complex structures on almost abelian solvmanifolds}
\label{sec:gcs-introduction-and-everyting}

In this section, we recall the basics of generalized complex structures and specialize to left-invariant generalized complex structures on almost abelian solvmanifolds.

\subsection{Generalized complex structures: a quick account} 
\label{sec:generatiliesoncgs}
Let $M$ be a smooth manifold of real dimension $2n$. Throughout, we let $T_M$ denote the tangent
bundle of $M$ and $T_M^\ast$ its cotangent bundle. We adopt the letters $X, Y$ for vector fields on
$M$ (smooth sections of $T_M$) and $\xi, \eta$ for $1$-forms on $M$ (smooth sections of $T_M^\ast$).
Notice that we can endow $T_M\oplus T_M^\ast$ with the split-signature symmetric pairing
    \begin{equation}\label{eq:Clifford-pairing}
    \langle X + \xi, Y + \eta \rangle := \frac{1}{2} \big( \xi(Y) + \eta(X) \big)
    \end{equation}
on smooth sections $X+ \xi,\, Y + \eta \in C^{\infty}(T_M \oplus T^{*}_M)$. We start by recalling the
concept of \emph{generalized complex structure}.

\medskip

First, recall that the bundle $\bigwedge^\bullet T^*_M$ carries the structure of a \emph{Clifford module}
over $T_M \oplus T^{*}_M$ \cite{Hitchin2003}*{Section 3.2}, where sections act by
$$(X+\xi) \cdot \rho = \iota_X \rho + \xi \wedge \rho$$
for $X+\xi \in C^{\infty}(T_M \oplus T^{*}_M)$ and $\rho \in C^{\infty}(\bigwedge^\bullet T^*_M)$. Here,
of course, the Clifford bundle $\mathrm{Cl}\left(T_M\oplus T_M^\ast,\langle\cdot,\cdot\rangle\right)$
is taken with respect to the quadratic form induced by $\langle\cdot,\cdot\rangle$ defined in Equation
\eqref{eq:Clifford-pairing}; with these conventions one has
$(X+\xi)\cdot(X+\xi)\cdot\rho = \langle X+\xi, X+\xi\rangle\, \rho$. Both the pairing and the Clifford
action are extended $\mathbb C$-bilinearly to complexifications whenever needed.

Second, consider the \emph{Courant bracket} on smooth sections of $T_M \oplus T^{*}_M$
\cite{Hitchin2003}*{Section 2}:
\begin{equation}\label{eq:Courant-Bracket}
[[X+\xi, Y+\eta]] = [X,Y] + \mathcal{L}_X \eta - \mathcal{L}_Y \xi
- \frac{1}{2}\d (\iota_X \eta - \iota_Y \xi).
\end{equation}

 \begin{defin}\label{def:gce}
     A \emph{generalized complex structure} on $M$ is a smooth section $\mathcal J$ of the endomorphism
     bundle $\mathrm{End}(T_M\oplus T_M^\ast)$ such that
     \begin{itemize}
         \item[(a)] $\mathcal J^2=-\mathrm{Id}$, thus
         $\left(T_M\oplus T_M^\ast\right)\otimes_{\mathbb R}\mathbb C=L\oplus \bar L$, where
         $L:=\ker\left(\mathcal J-\mathbf{i}\,\right)$;
         \item[(b)] $[[C^\infty(L),C^\infty(L)]]\subset C^\infty(L)$, after the obvious extension of the
         Courant bracket to the complexification
         $\left(T_M\oplus T_M^\ast\right)\otimes_{\mathbb R}\mathbb C$;
         \item[(c)] $\langle \mathcal J\cdot,\mathcal J\cdot\rangle=\langle\cdot,\cdot\rangle$.
         \label{item:orthogonality}
     \end{itemize}
 \end{defin}

Definition \ref{def:gce}\text{(a)} implies that $L \cap \overline{L} = \{0\}$, while Definitions
\ref{def:gce}\text{(a)} and \ref{def:gce}\text{(c)} together imply that $L$ is isotropic: for
$u,v \in L$ one has $\langle u,v\rangle = \langle \mathcal Ju,\mathcal Jv\rangle = -\langle u,v\rangle$.
Since $\operatorname{rk}_{\mathbb C}L = 2n$ and $\langle\cdot,\cdot\rangle$ has signature $(2n,2n)$, the
sub-bundle $L$ is in fact \emph{maximal} isotropic. As such, $L$ is the \emph{Clifford annihilator}
$$L = \{v \in (T_M \oplus T^{*}_M) \otimes_{\mathbb R} \mathbb{C} \mid v \cdot K_L = 0 \}$$
of a unique line sub-bundle $K_L \subset \bigwedge^\bullet T^*_M \otimes_{\mathbb R} \mathbb{C}$. A line
whose annihilator is maximal isotropic is, by definition, a \emph{pure} spinor line, and we call $K_L$
\emph{the canonical line bundle of $\mathcal J$}. We now move towards giving a local description of
$K_L$.

\medskip

Being a line bundle, $K_L$ admits, around each point of $M$, a smooth local generator
$\rho \in C^{\infty}(U, \bigwedge^\bullet T^*_M \otimes_{\mathbb R} \mathbb{C})$. Following
\cites{Gualtieri2011, Cavalcanti2004}, purity of $\rho$ means that at each point $x \in U$ one can write
  $$\rho_x = \exp(B+\mathbf{i}\omega)\wedge\Omega,$$
where $B,\omega \in \bigwedge^2 T^*_xM$ and $\Omega$ is \emph{decomposable}, i.e.,
$\Omega = \theta_1 \wedge \cdots \wedge \theta_k$ with $\theta_1, \dots, \theta_k \in T^*_xM \otimes
\mathbb{C}$ linearly independent. At type zero, $\Omega$ is a nonzero complex scalar, which is $1$ for a normalized generator. We remark that $k$
actually depends on the point of $M$ and need not be constant everywhere; accordingly, $B$, $\omega$ and
$\Omega$ can be chosen to depend smoothly on $x$ only where $k$ is locally constant. We will come back to
this in a moment. For now, we observe that the constraint $L \cap \overline{L} = \{0\}$ is equivalent to
the requirement that the pure form $\rho$ be \emph{non-degenerate}:
  $$\omega^{n-k} \wedge \Omega \wedge \overline{\Omega} \neq 0.$$
Note that such a condition does not involve $B$ and forces $k \in \{0,\ldots, n\}$. Following
\cite{Gualtieri2011}, \emph{integrability} of $\mathcal{J}$ (Definition \ref{def:gce}\text{(b)}) is
equivalent to requiring that for any local generator $\rho$ of $K_L$ there exists a local section
$X + \xi \in C^\infty((T_M \oplus T^{*}_M) \otimes_{\mathbb R} \mathbb{C})$ such that
$$\d\rho = (X + \xi) \cdot \rho.$$
Thus, we see that understanding generalized complex structures passes through understanding the equations
constraining such local generators. From now on, we say that $\rho$ is a \emph{pure spinor}.

\begin{defin}\label{def:agce}
    An \emph{almost generalized complex structure} on $M$ is a smooth section $\mathcal J$ of
    $\mathrm{End}(T_M\oplus T_M^\ast)$ satisfying items \text{(a)} and \text{(c)} of
    Definition \ref{def:gce}. A generalized complex structure is thus an integrable almost
    generalized complex structure.
\end{defin}

\begin{lemma}\label{lem:spinor-determines-J}
    Let $K\subset \bigwedge^\bullet T^*_M\otimes_{\mathbb R}\mathbb C$ be a line sub-bundle which is
    pure and non-degenerate at every point, and set $L:=\{v \in \left(T_M\oplus T_M^\ast\right)\otimes_{\mathbb R}\mathbb C \mid v\cdot K=0\}$. Then
    $$\left(T_M\oplus T_M^\ast\right)\otimes_{\mathbb R}\mathbb C=L\oplus\bar L.$$
    Let
    $$\mathrm{pr}_L,\ \mathrm{pr}_{\bar L}:
    \left(T_M\oplus T_M^\ast\right)\otimes_{\mathbb R}\mathbb C\longrightarrow
    \left(T_M\oplus T_M^\ast\right)\otimes_{\mathbb R}\mathbb C$$
    denote the projections onto $L$ and onto $\bar L$ associated with this splitting, so that
    $v=\mathrm{pr}_L(v)+\mathrm{pr}_{\bar L}(v)$ for every
    $v\in C^\infty\big((T_M\oplus T_M^\ast)\otimes_{\mathbb R}\mathbb C\big)$. With this notation,
    $$\mathcal J:=\mathbf i\,\mathrm{pr}_L-\mathbf i\,\mathrm{pr}_{\bar L}$$
    is the unique almost generalized complex structure on $M$ with $K_L=K$. Conversely, an almost generalized complex structure determines $K_L$, and hence determines any local generator $\rho$ of $K_L$ up to multiplication by a nowhere-vanishing smooth complex function.
\end{lemma}

\begin{proof}
    Non-degeneracy of $K$ means $L\cap\bar L=\{0\}$, which together with
    $\operatorname{rk}_{\mathbb C}L=2n$ gives the splitting. The endomorphism $\mathcal J$ is real
    because $\overline{\mathrm{pr}_L}=\mathrm{pr}_{\bar L}$, and $\mathcal J^2=-\mathrm{Id}$ by
    construction. Finally, write $u=u_L+u_{\bar L}$ and $v=v_L+v_{\bar L}$. Using that both
    $L$ and $\bar L$ are isotropic,
    $$\langle\mathcal Ju,\mathcal Jv\rangle=\langle u_L,v_{\bar L}\rangle+\langle u_{\bar L},v_L\rangle
    =\langle u,v\rangle,$$
    so Definition \ref{def:gce}\text{(c)} holds. Uniqueness follows because $L$ is the $\mathbf i$-eigenbundle of $\mathcal J$.
\end{proof}
\begin{lemma} \label{lem:integ_equiv}
    Let $U\subseteq M$ be an open set and let $\rho=\exp(B+\mathbf i\omega) \wedge \Omega$ be a
    non-degenerate pure spinor on $U$ where $B,\omega \in C^\infty(U,\bigwedge^2 T^*_M)$ and
    $\Omega = \theta_1 \wedge \cdots \wedge \theta_k$. Assume that 
    $\theta_1,\dots,\theta_k\in C^\infty(U,T^*_M\otimes\mathbb C)$ are pointwise linearly independent. Write $\beta:=B+\mathbf i\omega$.
    Then the almost generalized complex structure determined by $\rho$
    (Lemma \ref{lem:spinor-determines-J}) is integrable on $U$ if and only if there exists a section
    $X+\xi \in C^\infty(U,(T_M \oplus T_M^*)\otimes_{\mathbb R}\mathbb{C})$ such that
    \begin{itemize}
        \item[(a)] $\iota_X\Omega = 0$,
        \item[(b)] $\d\Omega = \alpha\wedge\Omega$, where $\alpha:=\xi + \iota_X \beta$,
        \item[(c)] $\d\beta \wedge \Omega = 0$.
    \end{itemize}
\end{lemma}
\begin{proof}
    Since $\beta$ has even degree, $\d (\exp{\beta})=\d\beta\wedge \exp{(\beta)}$ and
    $\iota_X \exp{(\beta)}=(\iota_X\beta)\wedge \exp{(\beta)}$. Moreover, both $\d\beta$ and $\iota_X\beta$
    commute with $\exp{(\beta)}$. Hence,
    $$\d\rho = \exp{(\beta)}\wedge\big(\d\Omega + \d\beta\wedge\Omega\big),\qquad
      (X+\xi)\cdot\rho = \exp{(\beta)}\wedge\big(\iota_X\Omega + \alpha\wedge\Omega\big).$$
The operator $\exp{(\beta)}\wedge(\cdot)$ is invertible with inverse $\exp({-\beta})\wedge(\cdot)$. The
    integrability condition $\d\rho=(X+\xi)\cdot\rho$ is equivalent to
    $$\d\Omega + \d\beta \wedge \Omega = \iota_X\Omega + \alpha\wedge \Omega.$$
    Because $\Omega$ is homogeneous of degree $k$ on $U$, the four terms above are homogeneous of
    degrees $k+1$, $k+3$, $k-1$ and $k+1$, respectively. Comparing the components of degree $k-1$,
    $k+1$ and $k+3$ yields \textit{(a)}, \textit{(b)}, and \textit{(c)}.
\end{proof}

\medskip

Having fixed a generalized complex structure, consider the image of the $\mathbf i$-eigenbundle $L$ under the natural projection $$\pi_T \colon (T_M \oplus T^*_M) \otimes_{\mathbb R} \mathbb{C} \longrightarrow T_M \otimes_{\mathbb R} \mathbb{C}.$$ Let $U\subseteq M$ be an open set on which the pure spinor $\rho=\exp\beta\wedge\Omega$  is such that $\Omega=\theta_1\wedge\ldots\wedge\theta_k$. A covector $\xi$ lies in $L$ when
$\xi\wedge \exp\beta\wedge\Omega=0$, equivalently $\xi\wedge\Omega=0$. Thus $\ker(\pi_T|_L)=\operatorname{span}_{\mathbb C}\{\theta_1,\ldots,\theta_k\}$ has rank $k$. Thus, $E:=\pi_T(L)$ is a complex distribution of rank $2n-k$ on $U$. Moreover, as $\pi_T$ is surjective and
$\left(T_M \oplus T^*_M\right) \otimes_{\mathbb R} \mathbb{C} = L \oplus \overline{L}$, we obtain
$$E + \overline{E} = T_M \otimes_{\mathbb R} \mathbb{C},\qquad
\dim_{\mathbb C}\left(E\cap\overline E\right) = 2(n-k).$$

\begin{defin}
    We term the number $k$ above \emph{the local type of $\mathcal J$}, and we say that it is
    \emph{extremal} if $k\in\{0,n\}$.
\end{defin}

In Example \ref{ex:extremal-cases} below, we show how extremal local types are related to complex and symplectic structures.

\begin{example}\label{ex:extremal-cases}
Assume that $(M,\omega)$ is a symplectic manifold. We can think of
$\omega \in C^\infty(\bigwedge^2 T^*_M)$ as the induced isomorphism $\omega\colon T_M \to T^*_M$,
$X\mapsto \iota_X\omega$. Then
\[\mathcal{J}_\omega := \begin{pmatrix} 0 & -\omega^{-1} \\ \omega & 0 \end{pmatrix}\]
can be checked to define a generalized complex structure such that
$$ L_\omega = \{X - {\mathbf i}\,\omega(X) \mid X \in C^\infty(T_M \otimes_{\mathbb R} \mathbb{C})\},$$ whose integrability amounts to $\d\omega=0$. Notice that $\pi_T|_{L_\omega}$ is an
isomorphism onto $T_M \otimes_{\mathbb R}\mathbb{C}$, so that $E$ has rank $2n$, yielding a
generalized complex structure of local type $0$. Equivalently, $K_{L_\omega}$ is globally generated
by the pure spinor
$$ \rho = \exp(\mathbf{i}\omega), $$
whose lowest-degree component is the constant function $1$, that is, $\Omega=1$.

Likewise, assume that $(M, J)$ is a complex manifold. Then
$$ \mathcal{J}_J := \begin{pmatrix} -J & 0 \\ 0 & J^* \end{pmatrix}$$
can be checked to define a generalized complex structure on $M$ such that
$$ L_J = T_{0,1} \oplus T^*_{1,0}.$$
Here, integrability corresponds precisely to the integrability of $J$. Its projection is
$\pi_T(L_J)=T_{0,1}$, of rank $n=2n-k$, yielding a generalized complex structure of local type $n$.
Equivalently, $K_{L_J}$ is the canonical bundle $\bigwedge^n T^*_{1,0}$ of $(M,J)$, hence generated locally by a holomorphic volume form
$$ \rho = \Omega = \d z^1\wedge\cdots\wedge \d z^n, $$
which is decomposable of degree $k=n$.
\end{example}

The concept of generalized complex structure first appeared in the literature in \cite{Hitchin2003},
where a generalized notion of ``Calabi--Yau structure'' is defined. As we will see straight from the
definition, any symplectic manifold carries the structure to be defined next:

\begin{defin}
A \emph{generalized Calabi--Yau structure} on $M$ is a generalized complex structure whose canonical line has a global nowhere-vanishing closed generator. An \emph{invariant generalized Calabi--Yau structure} on a Lie group is one whose canonical line has a closed left-invariant generator. On $\Gamma\backslash G$, invariant forms mean forms
descending from left-invariant forms on $G$. This is stronger than requiring only that the generalized complex structure be invariant.
\end{defin}

\medskip

We have now collected the minimum background on the general theory of generalized complex structures
needed in this paper. We move towards recalling some concepts and fixing notation and terminology
related to almost abelian solvmanifolds. Then, in Subsection \ref{sec:gce-aasm}, we revisit part of the preceding
discussion in that setting.

\medskip

\subsection{Almost abelian solvmanifolds}

\subsubsection{Almost abelian Lie groups}

Let $G$ be a Lie group. We say that $G$ is \emph{almost abelian} if its Lie algebra $\mathfrak{g}$ is
almost abelian, that is, if $\mathfrak g$ admits an abelian ideal $\mathfrak{h}$ of codimension one.
Set $d:=\dim\mathfrak h=\dim\mathfrak g-1$. Choosing an element $e_0 \in\mathfrak g\setminus\mathfrak{h}$,
the Lie algebra $\mathfrak{g}$ decomposes as
$$\mathfrak{g}=\mathbb{R}e_0 \ltimes \mathfrak{h},$$
and, since $\mathfrak h$ is an abelian ideal, the whole bracket of $\mathfrak g$ is recorded in
$\ad_{e_0}|_{\mathfrak{h}}\in \mathrm{End}(\mathfrak h)$. Fixing a basis for $\mathfrak{h}$, this
endomorphism is represented by a real matrix $A$, so that $\mathfrak{g}$ is isomorphic to
$$\mathfrak{g}_A:= \mathbb{R}e_0 \ltimes_A \mathfrak{h},$$
whose Lie bracket $[\cdot,\cdot]$ is determined by $\ad_{e_0}(v)=[e_0,v]=Av$ for $v \in \mathbb{R}^d$,
where we identify $\mathfrak h$ with $\mathbb R^d$.

Notice that $A$ depends on the choices made. Indeed, every other element of
$\mathfrak g\setminus\mathfrak h$ is of the form $\lambda e_0+h$, with $\lambda\in\mathbb R^\ast$ and
$h\in\mathfrak h$. Replacing $e_0$ with it changes $A$ with $\lambda A$, since $\mathfrak h$ is
abelian. Likewise, changing the basis of $\mathfrak h$ replaces $A$ with $PAP^{-1}$, for some
$P\in \mathrm{GL}_d(\mathbb R)$. Moreover, the ideal $\mathfrak h$ itself need not be unique. Conversely, matrices related as above always produce isomorphic
Lie algebras and, interestingly enough, nothing else does:

\begin{prop}{\cite{Freibert2012}*{Proposition 2.9}}\label{Conjugate}
Two almost abelian Lie algebras $\mathfrak{g}_A$ and $\mathfrak{g}_B$, determined by
$A,B \in M(d, \mathbb{R})$ respectively, are isomorphic if and only if $A$ and $cB$ are conjugate,
for some $c \neq 0$.
\end{prop}
Proposition \ref{Conjugate} particularly shows that $A$ is well defined up to conjugation and
nonzero scaling, independently of every choice made above, including that of $\mathfrak h$.

\medskip

At the Lie group level, the simply connected Lie group $G_A$ with Lie algebra $\mathfrak{g}_A$ is realized as $$G_A = \mathbb{R} \ltimes_\phi \mathbb{R}^d,$$ where the real line acts on $\mathbb{R}^d$ by the matrix exponential
$\phi(t)=e^{tA}\in\mathrm{Aut}(\mathbb R^d)$.

\begin{example}
    Every almost abelian Lie group is $2$-step solvable, while $G_A$ is nilpotent if and only if $A$ is a nilpotent matrix \cite{Avetisyan2022}. 
\end{example}

In this subsection, we are interested in almost abelian solvmanifolds. Having recalled above the
concept of almost abelian Lie groups, we now briefly recall some basic terminology concerning
solvmanifolds.

\medskip

\subsubsection{Almost abelian solvmanifolds}
We start by recalling the following:

\begin{defin}
A \emph{solvmanifold} is a quotient space $\Gamma \backslash G$, where $G$ is a simply connected
solvable Lie group and $\Gamma \subset G$ is a discrete cocompact subgroup, called a \emph{lattice}.
When $G=G_A$ is almost abelian, we call $M=\Gamma \backslash G_A$ an \emph{almost abelian
solvmanifold}.
\end{defin}

A Lie group $G$ with Lie algebra $\mathfrak{g}$ is said to be \emph{unimodular} if its
left Haar measure is also right-invariant. In the case where $G$ is connected, this can be checked to be equivalent to
$\tr(\ad_X)=0$ for all $X \in \mathfrak{g}$. For $G_A$ almost abelian, this is equivalent to  $\tr(A)=0$: writing $X=te_0+v$ gives
$\tr(\ad_X)=t\,\tr(A)$. Finally, recall that if a Lie group admits a lattice, then it is
unimodular \cite{Milnor1976}*{Lemma 6.2}. Hence, $\tr(A)=0$ is a necessary condition for the existence of an almost abelian solvmanifold $\Gamma\backslash G_A$. 
Proposition \ref{prop:andrada} gives a sufficient condition:

\begin{prop}\label{cond_lattice}\cite{Andrada2023}*{Proposition 2.5}\label{prop:andrada}
    Let $G_A=\mathbb{R} \ltimes_\phi \mathbb{R}^d$ be a unimodular almost abelian Lie group, where $\phi(t)=e^{tA}$. Then $G_A$ admits a lattice if and only if there exists $t_0 \neq 0$ such that $e^{t_0 A}$ is conjugate to an invertible integer matrix in $\mathrm{SL}(d, \mathbb{Z})$. In this case, a lattice is given by $\Gamma=t_0\mathbb{Z} \ltimes_\phi P\mathbb{Z}^d$, with $P \in \mathrm{GL}(d,\mathbb{R})$ satisfying $P^{-1}e^{t_0A}P \in \mathrm{SL}(d,\mathbb{Z})$.
\end{prop}

\medskip

With the conventions above, the elements of $\Gamma\backslash G_A$ are the cosets $\Gamma g$. Thus,
every left-invariant tensor on $G_A$ is also $\Gamma$-invariant, thus descends 
to $\Gamma\backslash G_A$. In light of this, next, we describe the left-invariant generalized complex structures in $G_A$, which ultimately will lead to generalized complex
structures in almost abelian solvmanifolds.

\medskip

\subsection{Generalized complex structures on almost abelian solvmanifolds}
\label{sec:gce-aasm}

For a given Lie group $G$, left-translation $L_g$ by an element $g\in G$ induces, via its derivative,
left-translations of vectors and covectors on $G$:
$$g * (X_h + \xi_h):=(\d L_g)_hX_h+(L_{g^{-1}})^*_{gh}\xi_h\in
\left(T_{gh}G\right)\oplus\left(T^*_{gh}G\right).$$
Notice that $g*$ preserves the pairing \eqref{eq:Clifford-pairing}, since
$(L_{g^{-1}})^*_{gh}\xi_h$ evaluated on $(\d L_g)_hY_h$ is just $\xi_h(Y_h)$. A generalized complex
structure $\mathcal{J}$ on $G$ is said to be \emph{left-invariant} if, for any $g,h \in G$ and
$X_h+\xi_h\in \left(T_hG\right)\oplus\left(T^*_hG\right)$, we have
$$\mathcal{J}_{gh}(g * (X_h+\xi_h))=g * (\mathcal{J}_h(X_h+\xi_h)).$$
Of course, to fully characterize a left-invariant generalized complex structure it suffices to
understand it at the level of the Lie algebra $T_eG\cong \mathfrak g$: we identify the space of
left-invariant sections of $T_G \oplus T_G^*$ with $\mathfrak{g} \oplus \mathfrak{g}^*$ and notice that
$$\mathcal{J}_g(X+\xi)=g * \mathcal{J}_{e}(g^{-1} * (X+\xi)).$$
In particular, the local type of $\mathcal J$ is the same at every point, being the local type of
$\mathcal J_e$.

Following \cite{Barberis}, on left-invariant sections the Courant bracket simplifies to the Lie bracket
of the cotangent algebra $T^*\mathfrak{g}=\mathfrak g\ltimes_{\ad^*}\mathfrak g^*$:
$$[[X+\xi, Y+\eta]]_\mathfrak{g}=[X,Y]+\ad^*_X\eta - \ad^*_Y\xi.$$
Here, $\ad^*$ denotes the representation of $\mathfrak g$ on $\mathfrak g^{\ast}$ induced by the
coadjoint representation of $G$, namely $(\ad^*_X\eta)(Y)=-\eta([X,Y])$. Indeed, for left-invariant
$X$ and $\eta$ the function $\eta(X)$ is constant. Thus, $\mathcal L_X\eta=\iota_X\d\eta=\ad^*_X\eta$
and the term $\frac{1}{2}\d(\iota_X \eta - \iota_Y \xi)$ in \eqref{eq:Courant-Bracket} vanishes. 

\medskip

Based on the description of generalized complex structures through spinors, left-invariant generalized
complex structures can be readily treated following \cite{Cavalcanti2004}. A major simplification comes
from the fact that the local type is constant everywhere and that the canonical line bundle $K_L$ is
trivial. We can then choose a global trivialization
$$\rho= \exp(B+{\mathbf i}\omega)\wedge\Omega,$$
where $B,\omega \in {\bigwedge}^2 \mathfrak{g}^*$ and
$\Omega=\theta_1 \wedge\cdots\wedge \theta_k$ is decomposable, with
$\theta_1,\ldots,\theta_k\in\mathfrak g^*\otimes\mathbb C$ linearly independent. Furthermore, the
integrability condition needs only be posed for left-invariant sections
$X+\xi \in (\mathfrak{g} \oplus \mathfrak{g}^*) \otimes_{\mathbb{R}}\mathbb{C}$. 

\medskip

\section{Admissible types: existence and algebraic obstructions}

In this section, we investigate the admissible types of left-invariant generalized complex structures on almost abelian Lie groups and, consequently, on their solvmanifolds. We first provide a complete spectral classification for the diagonalizable case, including the criteria for generalized Calabi--Yau structures. For the general case, we introduce a constructive method via algebraic reduction and establish strict upper bounds on the type dictated by the Jordan structure of the defining matrix.

\subsection{Classification in the diagonal case}

Since the local type of a left-invariant generalized complex structure is constant, we shall simply call it its \emph{type}. To approach the classification question, recall from Proposition \ref{Conjugate} that isomorphism classes of almost abelian Lie algebras of dimension $d+1$ are in bijection with conjugacy classes of matrices in $M(d,\mathbb R)$, modulo multiplication by a nonzero real scalar. 

Below, we establish the complete classification of admissible types of left-invariant generalized complex structures on $2n$-dimensional almost abelian Lie groups with Lie algebra $\mathfrak{g}_A=\mathbb{R}e_0 \ltimes_A \mathfrak{h}$, where the adjoint action of $e_0$ on the abelian ideal $\mathfrak{h} \cong \mathbb{R}^{2n-1}$ is represented by a matrix $A$ which is diagonalizable over $\mathbb R$. In this case, Proposition \ref{Conjugate} allows us to assume that $A$ is diagonal, which we do from now on. The purpose of this subsection is to prove the following:

\begin{theo}\label{thm:os-meninos-vem-como}
    Let $G_A$ be a $2n$-dimensional almost abelian Lie group with Lie algebra
    $\mathfrak{g}_A= \mathbb{R}e_0 \ltimes_A \mathfrak{h}$, where $A$ is diagonal. Then $G_A$ admits a
    left-invariant generalized complex structure of type $k$ if and only if the diagonal entries of
    $A$, listed with multiplicity, can be reordered in one of the following forms:
    \begin{enumerate}[label=(\alph*)]
        \item $\mathrm{spec}(A) = \{\underbrace{\lambda_1, \lambda_1, \dots, \lambda_{k}, \lambda_{k}}_{k \text{ identical pairs}}, \mu, \underbrace{\zeta_1, -\zeta_1, \dots, \zeta_{n-k-1}, -\zeta_{n-k-1}}_{n-k-1 \text{ opposite pairs}} \}$, with $k \le n-1$;
        \item $\mathrm{spec}(A) = \{\underbrace{\lambda_1, \lambda_1, \dots, \lambda_{k-1}, \lambda_{k-1}}_{k-1 \text{ identical pairs}}, \mu, \underbrace{\zeta_1, -\zeta_1, \dots, \zeta_{n-k}, -\zeta_{n-k}}_{n-k \text{ opposite pairs}} \}$, with $k \ge 1$.
    \end{enumerate}
    Here, $\mathrm{spec}(A)$ is understood as a multiset of $2n-1$ elements.
\end{theo}
Since invariant generalized complex structures on $M=\Gamma\backslash G_A$ correspond to left-invariant
ones on $G_A$, we obtain at once the solvmanifold version:

\begin{cor}\label{cor:solv-version}
    Let $M=\Gamma\backslash G_A$ be a $2n$-dimensional almost abelian solvmanifold with $A$ diagonal.
    Then $M$ admits an invariant generalized complex structure of type $k$ if and only if
    $\mathrm{spec}(A)$ is as in Theorem \ref{thm:os-meninos-vem-como}. In this case, $\tr A=0$ and thus $\mu=-2\sum_i\lambda_i$.
\end{cor}
Theorem \ref{thm:os-meninos-vem-como} was inspired by the existing classification of left-invariant
complex (type $n$) and symplectic (type $0$) structures on almost abelian Lie groups
\cite{Arroyo2025}, and recovers it as the extremal cases:

\begin{theo}\label{Sym_Com_Diag}\cite{Arroyo2025}
    Let $G_A$ be a $2n$-dimensional almost abelian Lie group whose Lie algebra
    $\mathfrak{g}_A=\mathbb{R}e_0 \ltimes_A \mathfrak{h}$ is determined by a diagonal matrix $A$. Then
    $G_A$ admits a left-invariant
    \begin{enumerate}[label=(\alph*)]
        \item symplectic structure if and only if $A$ has $n-1$ disjoint pairs of opposite eigenvalues;
        \item complex structure if and only if $A$ has $n-1$ disjoint pairs of identical eigenvalues.
    \end{enumerate}
\end{theo}
Indeed, taking $k=0$ in Theorem \ref{thm:os-meninos-vem-como}(a) and $k=n$ in Theorem \ref{thm:os-meninos-vem-como}(b) yields
the two conditions above. Before we dive into the proof of Theorem
\ref{thm:os-meninos-vem-como}, let us explore some consequences.
\medskip

The first is immediate:

\begin{cor}\label{cor:adjacent-types}
    Let $G_A$ be an almost abelian Lie group with Lie algebra
    $\mathfrak{g}_A = \mathbb{R}e_0 \ltimes_A \mathfrak{h}$, where $A$ is diagonal. If $G_A$
    admits a left-invariant generalized complex structure of type $k$, then $G_A$ also admits a
    left-invariant generalized complex structure of type $k+1$ or of type $k-1$.
\end{cor}
\begin{proof}
   Notice that Theorem \ref{thm:os-meninos-vem-como}(a) for $k$ and Theorem \ref{thm:os-meninos-vem-como}(b) for $k+1$ describe the same multiset. Hence, if
    $\mathrm{spec}(A)$ is as in item (a), then $G_A$ admits a structure of type $k+1$. If it is as
    in item (b), then $G_A$ admits one of type $k-1$.
\end{proof}
Theorem \ref{thm:os-meninos-vem-como}, in its own right, allows us to construct almost abelian Lie
groups admitting only left-invariant generalized complex structures of \emph{intermediate type}, i.e.,
non-extremal, as the following example explains.

\begin{example} \label{ex:int_types_only}
    For any dimension $2n\ge6$ (i.e., $n\ge3$), consider a diagonal matrix $A$ with
    \[ \mathrm{spec}(A)=\{\lambda_1, \lambda_1, \ldots, \lambda_{n-2}, \lambda_{n-2}, \mu, \zeta, -\zeta\}, \]
    where $\lambda_i$ and $\zeta$ are non-zero. To avoid any additional matchings, we require $\lambda_i \neq \pm \lambda_j$ for $i \neq j$, $\lambda_i \neq \pm \zeta$, and that the singleton $\mu$ is distinct from all $\pm \lambda_i$ and $\pm \zeta$.

    Theorem \ref{thm:os-meninos-vem-como} says that the almost abelian Lie group associated with this
    matrix admits left-invariant generalized complex structures of type $n-2$, by item (a) with $k=n-2$,
    and of type $n-1$, by item (b) with $k=n-1$, in accordance with Corollary \ref{cor:adjacent-types}.
    However, by Theorem \ref{Sym_Com_Diag}, it is strictly obstructed from admitting a left-invariant complex or symplectic structure: the set contains exactly $n-2$ identical pairs (which rules out complex structures since $n-1$ are needed) and exactly one opposite pair (which rules out symplectic structures since $n-1 \ge 2$ are needed).
\end{example}

We now provide a concrete arithmetic realization of the spectral configuration from Example \ref{ex:int_types_only} for $n=5$. This explicit construction ensures the existence of a lattice on the corresponding almost abelian Lie group.\footnote{The authors thank B. Brienza for posing this question.}

\begin{example} \label{ex:dim10-lattice}
 Let $f(x)=x^3-5x^2+6x-1$ and $g(x)=x^2-3x+1$. The values $f(0)=-1$, $f(1)=1$, $f(2)=-1$, and $f(4)=7$ give three distinct positive roots in $(0,1),(1,2),(2,4)$. Neither $1$ nor $-1$ is a root, so $f$ is irreducible over $\mathbb Q$. The roots of $g$ are $z=(3+\sqrt5)/2$ and $z^{-1}$. Write $y_1,y_2,y_3$ for the roots of $f$, set
$\lambda_i=\log y_i$, $\zeta=\log z$, $\mu=0$, and put
\[
A=\operatorname{diag}(\lambda_1,\lambda_1,\lambda_2,\lambda_2,
\lambda_3,\lambda_3,\mu,\zeta,-\zeta).
\]
The identity $y_1y_2y_3=1$ gives $\sum_i\lambda_i=0$ and $\operatorname{tr}A=0$. If $i\ne j$ and $y_i y_j=1$, the remaining root would be $1$. If $i=j$ and $y_i^2=1$, positivity also gives $y_i=1$. Both are impossible. The remainder of $f$ modulo $g$ is $1-x$, so a common root would be $1$, whereas $g(1)=-1$. Thus $\lambda_i\ne\pm\lambda_j$ for $i\ne j$,
$\lambda_i\ne0$, and $\lambda_i\ne\pm\zeta$.

    In this case, $e^{A}$ is diagonal with eigenvalues $y_1, y_1, y_2, y_2, y_3, y_3, 1, z, z^{-1}$, so that its characteristic and minimal polynomials are
    \begin{align}
        P(x) &= \left(x^3 - 5x^2 + 6x - 1\right)^2 \left(x^2 - 3x + 1\right)(x - 1)\in\mathbb Z[x], \\
        m(x) &= \left(x^3 - 5x^2 + 6x - 1\right)\left(x^2 - 3x + 1\right)(x - 1)\in\mathbb Z[x].
    \end{align}
    As $e^A$ is diagonalizable, its minimal polynomial is $m$ and its invariant factors are $f_1(x)=x^3-5x^2+6x-1$ and $f_2(x)=m(x)$, which satisfy $f_1\mid f_2$ and $f_1f_2=P$. Both lie in $\mathbb Z[x]$, so the rational canonical form of $e^A$, namely the block-diagonal matrix $C=C_{f_1}\oplus C_{f_2}$ of the companion matrices of $f_1$ and $f_2$, has integer entries. Since $\det C_{f}=(-1)^{\deg f}f(0)$ for any monic $f$, we get $\det C=(-1)^3f_1(0)\cdot(-1)^6f_2(0)=1$, so that $C\in\mathrm{SL}(9,\mathbb Z)$ and in particular $C^{-1}$ is integral as well. Finally, $e^A$ and $C$ are conjugate over $\mathbb R$, so the condition of Proposition \ref{prop:andrada} holds for $t_0=1$, providing a lattice $\Gamma<G_A$.

    The quotient $M=\Gamma\backslash G_A$ admits invariant structures of types $3$ and $4$, and no invariant structure of extremal type. Both constructions have closed invariant generators, since their values of $\Lambda$ are $\sum_i\lambda_i=0$ and $\mu+\sum_i\lambda_i=0$. This does not exclude non-invariant complex or symplectic structures on $M$.
\end{example}

Since we have classified left-invariant generalized complex structures on almost abelian Lie groups with $A$ diagonal, we pursue a criterion for left-invariant generalized Calabi--Yau structures. We obtain the
corresponding classification:

\begin{cor} \label{cor-CY-diag}
    Let $G_A$ be a $2n$-dimensional almost abelian Lie group with Lie algebra
    $\mathfrak{g}_A= \mathbb{R}e_0 \ltimes_A \mathfrak{h}$, where $A$ is diagonal. Then $G_A$
    admits a left-invariant generalized Calabi--Yau structure of type $k$ if and only if the diagonal
    entries of $A$, listed with multiplicity, can be reordered in one of the following forms:
    \begin{enumerate}[label=(\alph*)]
        \item $\mathrm{spec}(A) = \{\underbrace{\lambda_1, \lambda_1, \dots, \lambda_{k}, \lambda_{k}}_{k \text{ identical pairs}}, \mu, \underbrace{\zeta_1, -\zeta_1, \dots, \zeta_{n-k-1}, -\zeta_{n-k-1}}_{n-k-1 \text{ opposite pairs}} \}$, with $k \le n-1$ and $\sum_{i=1}^k \lambda_i=0$;
        \item $\mathrm{spec}(A) = \{\underbrace{\lambda_1, \lambda_1, \dots, \lambda_{k-1}, \lambda_{k-1}}_{k-1 \text{ identical pairs}}, \mu, \underbrace{\zeta_1, -\zeta_1, \dots, \zeta_{n-k}, -\zeta_{n-k}}_{n-k \text{ opposite pairs}} \}$, with $k \ge 1$ and $\mu + \sum_{i=1}^{k-1} \lambda_i=0$.
    \end{enumerate}
\end{cor}

\begin{cor}
    Let $M=\Gamma\backslash G_A$ be a $2n$-dimensional almost abelian solvmanifold with $A$ diagonal.
    Then $M$ admits an invariant generalized Calabi--Yau structure of type $k$ if and only if
    $\mathrm{spec}(A)$ is as in Corollary \ref{cor-CY-diag}. In this case, $\mathrm{tr}A=0$ and thus $\mu=0$.
\end{cor}

\medskip

We now turn to the proof of the above results, done via explicit computations using
pure spinors. We fix, herein, the following notation: $G_A$ stands for an almost abelian
Lie group of dimension $2n$ with associated Lie algebra
$\mathfrak{g}_A = \mathbb{R} e_0 \ltimes_A \mathfrak{h}$, where $\mathfrak{h} \cong \mathbb{R}^{2n-1}$
is an abelian ideal. We fix a basis $\{e_1, \dots, e_{2n-1}\}$ for $\mathfrak{h}$ and denote the dual
basis of $\mathfrak{g}_A^*$ by $\{e^0, e^1, \dots, e^{2n-1}\}$. In this basis, we write
$A=(a_{ij}) \in M(2n-1,\mathbb R)$ for the matrix of $\ad_{e_0}|_{\mathfrak h}$, so that the Lie bracket
of $\mathfrak{g}_A$ is determined by $[e_0, e_i] = \sum_{j=1}^{2n-1} a_{ji} e_j$.

We identify the left-invariant forms on $G_A$ with $\bigwedge^\bullet\mathfrak{g}_A^*$, and
we denote by $\d$ the restriction of the de Rham differential to them, which is determined by the
Maurer--Cartan equations
$$\d\xi(X,Y) = -\xi([X,Y]),\qquad \xi \in \mathfrak{g}_A^*,\ X,Y \in \mathfrak{g}_A .$$
Since $[\mathfrak g_A,\mathfrak g_A]\subseteq\mathfrak h=\ker e^0$ and $\mathfrak h$ is abelian, a
direct computation gives
$$\d e^0 = 0\text{ and }\d e^p = -e^0 \wedge A^T e^p\text{ for }p \in \{1, \dots, 2n-1\},$$
where $A^T\colon\mathfrak h^*\to\mathfrak h^*$ is the dual map of $\ad_{e_0}|_{\mathfrak h}$, whose
matrix in the dual basis is the transpose of $A$, that is, $A^Te^p=\sum_i a_{pi}e^i$. We remark that
$A^T=-\ad^*_{e_0}|_{\mathfrak h^*}$, with $\ad^*$ as fixed above.

\begin{defin}\label{def:Psi}
    Let $\Psi$ be the unique degree-preserving derivation of
    $\bigwedge^\bullet \mathfrak{g}_A^*\otimes_{\mathbb R}\mathbb C$ such that $\Psi(e^0)=0$ and
    $\Psi|_{\mathfrak h^*}=A^T$. Explicitly, for a decomposable $p$-form
    $\alpha = \eta_1 \wedge \cdots \wedge \eta_p$,
    $$\Psi(\eta_1 \wedge \cdots \wedge \eta_p) = \sum_{j=1}^p \eta_1 \wedge \cdots \wedge
    (\Psi\eta_j) \wedge \cdots \wedge \eta_p .$$
    In particular $\Psi\left(e^0\wedge\gamma\right)=e^0\wedge\Psi(\gamma)$ and $\Psi$ preserves
    $\bigwedge^\bullet \mathfrak{h}^*\otimes_{\mathbb R}\mathbb C$.
\end{defin}

\begin{prop}\label{prop:ext-derivative}
      The differential $\d$ defined above is given by $\d\alpha = -e^0\wedge\Psi(\alpha)$, for every
    $\alpha \in \bigwedge^\bullet \mathfrak{g}_A^*\otimes_{\mathbb R}\mathbb C$. In particular,
    $\d\left(e^0\wedge\gamma\right)=0$ for every
    $\gamma \in \bigwedge^\bullet \mathfrak{g}_A^*\otimes_{\mathbb R}\mathbb C$.
\end{prop}
\begin{proof}
Put $D\alpha=-e^0\wedge\Psi\alpha$. If $\deg\alpha=p$, then
\[
\begin{aligned}
D(\alpha\wedge\beta)
&=-e^0\wedge\Psi\alpha\wedge\beta-e^0\wedge\alpha\wedge\Psi\beta\\
&=D\alpha\wedge\beta+(-1)^p\alpha\wedge D\beta.
\end{aligned}
\]
Hence $D$ is a degree-one derivation. It agrees with $\d$ on
constants and basis one-forms, so agrees everywhere.
Finally $\Psi(e^0\wedge\gamma)=e^0\wedge\Psi\gamma$ gives
$\d(e^0\wedge\gamma)=0$.
\end{proof}
Finally, notice that when $A=\mathrm{diag}(\lambda_1,\ldots,\lambda_{2n-1})$ is diagonal, so is
$A^T=A$, and $\Psi$ acts diagonally on the induced basis of
$\bigwedge^\bullet \mathfrak{h}^*$:
$$\Psi\left(e^{i_1}\wedge\cdots\wedge e^{i_p}\right)
=\left(\lambda_{i_1}+\cdots+\lambda_{i_p}\right)e^{i_1}\wedge\cdots\wedge e^{i_p},$$
so that $\d\left(e^{i_1}\wedge\cdots\wedge e^{i_p}\right)
=-\left(\lambda_{i_1}+\cdots+\lambda_{i_p}\right)e^0\wedge e^{i_1}\wedge\cdots\wedge e^{i_p}$. This is
what makes the computations below tractable.

\medskip

For $k=0$, normalize $\Omega=1$ and set $W=0$, $\Omega_{ab}=1$, and $\Omega_0=0$. Integrability then means $\d\beta=0$, equivalently $\Psi(\beta_{ab})=0$, and the invariant spinor is closed. Arguments involving $\bigwedge^{k-1}W$ below are for $k\ge1$.

For $k\ge1$, consider a non-degenerate pure spinor
$\rho=\exp(\beta)\wedge\Omega$, where $\beta:=B+\mathbf i \omega$ with $B, \omega \in \bigwedge^2\mathfrak{g}^*_A$ real, and
$\Omega=\theta_1 \wedge\cdots\wedge\theta_k$
is decomposable, with $\theta_1,\ldots,\theta_k \in \mathfrak{g}^*_A\otimes_{\mathbb R}\mathbb{C}$
linearly independent. Writing $\theta_i=c_ie^0+\eta_i$, with $c_i \in \mathbb{C}$ and
$\eta_i\in \mathfrak{h}^* \otimes_{\mathbb R}\mathbb{C}$, and using that
$\bigwedge^\bullet\mathfrak g_A^*=\bigwedge^\bullet\mathfrak h^*\oplus\,e^0\wedge
\bigwedge^\bullet\mathfrak h^*$, we can uniquely decompose $\Omega$ as
$$\Omega=e^0\wedge\Omega_0+\Omega_{ab}.$$
Indeed, expanding the product $\theta_1\wedge\cdots\wedge\theta_k$ and setting
$\Omega_{\hat\jmath}:=\eta_1\wedge\cdots\wedge\widehat{\eta_j}\wedge\cdots\wedge\eta_k$, one gets
\begin{equation}\label{eq:OmegaabandOmega0}\Omega_{ab}=\eta_1 \wedge\cdots\wedge\eta_k\in\textstyle\bigwedge^{k} \mathfrak{h}^*
\otimes_{\mathbb R}\mathbb{C},\qquad
\Omega_{0}=\sum_{j=1}^{k}(-1)^{j-1}c_j\,\Omega_{\hat\jmath}\in\textstyle\bigwedge^{k-1} \mathfrak{h}^*
\otimes_{\mathbb R}\mathbb{C}.\end{equation}
In particular, the purely abelian component $\Omega_{ab}$ is again decomposable, and both components
are built out of the $\eta_j$ alone. Moreover, $\Omega_{ab}\neq0$: otherwise $\Omega=e^0\wedge\Omega_0$
and hence $\Omega\wedge\bar\Omega=0$, since $e^0$ is real, contradicting the non-degeneracy of $\rho$.
Setting $W:=\mathrm{span}_{\mathbb{C}}\{\eta_1, \dots, \eta_k\}$, this says that
$\eta_1,\ldots,\eta_k$ are linearly independent, so that $\dim_{\mathbb C}W=k$, the forms
$\Omega_{\hat 1},\ldots,\Omega_{\hat k}$ are a basis of $\bigwedge^{k-1}W$, thus
$\Omega_0\in\bigwedge^{k-1}W$, and $\Omega_{ab}$ spans the line $\bigwedge^{k}W$. Analogously, we write
$$\beta=\beta_{ab}+e^0\wedge\beta_0,\qquad \omega=\omega_{ab}+e^0\wedge\omega_0,$$
with $\beta_{ab},\omega_{ab}\in\bigwedge^2\mathfrak h^*$ and
$\beta_0,\omega_0\in\bigwedge^1\mathfrak h^*$, up to complexification in the case of $\beta$. This
yields the following.

\begin{lemma} \label{lemma-integ-equiv_new}
   A non-degenerate pure spinor $\rho = \exp(\beta) \wedge \Omega$ defines a left-invariant generalized complex structure on $G_A$ if and only if its components $\beta_{ab}$ and $\Omega_{ab}$, together with the subspace $W = \mathrm{span}_{\mathbb{C}}\{\eta_1, \dots, \eta_k\}$ associated to \eqref{eq:OmegaabandOmega0}, satisfy
    \begin{enumerate}
        \item[(a)] $W$ is invariant under $A^T$;
        \item[(b)] $\Psi(\beta_{ab}) \wedge \Omega_{ab} = 0$.
    \end{enumerate}
    In this case, we additionally have $\Psi(\Omega_{ab}) = \tr\left(A^T|_W\right)\Omega_{ab}$.
\end{lemma}
\begin{proof}
    Assume first that $\rho$ defines a left-invariant generalized complex structure. By Lemma \ref{lem:integ_equiv}, integrability requires $\d\Omega = \alpha \wedge \Omega$ for the complex $1$-form $\alpha=\xi+\iota_X\beta$, and also $\d\beta \wedge \Omega = 0$. 
    
    On the other hand, by Proposition \ref{prop:ext-derivative} we have $\d(e^0\wedge\Omega_0)=0$, so that $\d\Omega = -e^0 \wedge \Psi(\Omega_{ab})$. Equating both expressions for $\d\Omega$ yields
    \begin{equation}\label{eq:auxiliary}
        -e^0 \wedge \Psi(\Omega_{ab}) = \alpha \wedge (e^0 \wedge \Omega_0 + \Omega_{ab}).
    \end{equation}
    Wedging this equation with $e^0$ on the left leaves
    $$0 = e^0 \wedge \alpha \wedge \Omega_{ab}.$$
    Write $\alpha = a_0 e^0 + \gamma$, with $a_0 \in \mathbb{C}$ and $\gamma \in \mathfrak{h}^* \otimes_{\mathbb R}\mathbb{C}$. As $e^0\wedge(\cdot)$ is injective on $\bigwedge^\bullet\mathfrak h^*\otimes_{\mathbb R}\mathbb{C}$, the above reads $\gamma\wedge\Omega_{ab}=0$ which, $\Omega_{ab}$ being decomposable and non-vanishing, forces $\gamma \in W$.

    Substituting $\alpha$ back into Equation \eqref{eq:auxiliary} gives
    $$-e^0 \wedge \Psi(\Omega_{ab}) = a_0 e^0 \wedge \Omega_{ab} - e^0 \wedge \gamma \wedge \Omega_0 + \gamma \wedge \Omega_{ab}.$$
    Since $\gamma \in W$, we have $\gamma \wedge \Omega_{ab} = 0$, and matching the $e^0$ components on both sides yields
    $$\Psi(\Omega_{ab}) = -a_0 \Omega_{ab} + \gamma \wedge \Omega_0.$$
    Because $\gamma \in W$ and $\Omega_0 \in \bigwedge^{k-1}W$, their wedge product lies in $\bigwedge^k W$, which is the line spanned by $\Omega_{ab}$. Consequently, $\Psi(\Omega_{ab}) = \Lambda \Omega_{ab}$ for some $\Lambda \in \mathbb{C}$. Expanding via the derivation property,
    $$\sum_{j=1}^k \eta_1 \wedge \cdots \wedge A^T \eta_j \wedge \cdots \wedge \eta_k = \Lambda\,\eta_1\wedge\cdots\wedge\eta_k.$$

    Let $U$ be a complementary subspace, so that $\mathfrak{h}^* \otimes_{\mathbb R} \mathbb{C} = W \oplus U$, and decompose $A^T \eta_j = v_j + u_j$ with $v_j \in W$ and $u_j \in U$. Substituting this into the sum above, both the right-hand side and all the terms involving some $v_j$ lie in $\bigwedge^k W$, whereas
    $$\sum_{j=1}^k \eta_1 \wedge\cdots\wedge u_j \wedge\cdots\wedge \eta_k = \sum_{j=1}^k (-1)^{k-j}\,\Omega_{\hat\jmath}\wedge u_j \in \textstyle\bigwedge^{k-1}W\wedge U.$$
    Since $\bigwedge^k W\cap \left(\bigwedge^{k-1}W\wedge U\right)=\{0\}$, this last sum must vanish. But $\bigwedge^{k-1}W\wedge U\cong \bigwedge^{k-1}W\otimes U$ and $\Omega_{\hat 1},\ldots,\Omega_{\hat k}$ is a basis of $\bigwedge^{k-1}W$, so that $u_j = 0$ for every $j$. Thus $A^T \eta_j = v_j \in W$ for every $j$, proving that $W$ is invariant under $A^T$. 
    
    This invariance allows us to write $A^T\eta_j = \sum_{i=1}^k C_{ij} \eta_i$ for $C = A^T|_W$. Substituting this into the expansion of $\Psi(\Omega_{ab})$ yields
    $$ \Psi(\Omega_{ab}) = \sum_{j=1}^k \eta_1 \wedge \cdots \wedge \left(\sum_{i=1}^k C_{ij}\eta_i\right) \wedge \cdots \wedge \eta_k = \left(\sum_{j=1}^k C_{jj}\right) \Omega_{ab} = \tr\left(A^T|_W\right)\Omega_{ab}, $$
    since all terms with $i \neq j$ vanish by repetition. 
    
    Furthermore, the integrability condition $\d\beta \wedge \Omega = 0$ together with $\d\beta = -e^0 \wedge \Psi(\beta_{ab})$ implies
    $$ 0 = -e^0 \wedge \Psi(\beta_{ab}) \wedge (e^0 \wedge \Omega_0 + \Omega_{ab}) = -e^0 \wedge (\Psi(\beta_{ab}) \wedge \Omega_{ab}). $$
    By the injectivity of $e^0 \wedge (\cdot)$, this forces $\Psi(\beta_{ab}) \wedge \Omega_{ab} = 0$, giving item (b).

    Conversely, assume (a) and (b) hold. Because of (a), the calculation $\Psi(\Omega_{ab}) = \tr\left(A^T|_W\right)\Omega_{ab}$ performed above remains valid. We verify the integrability conditions of Lemma \ref{lem:integ_equiv} by choosing $X = 0$ and $\xi = -\tr\left(A^T|_W\right)e^0$. Condition (a) of Lemma \ref{lem:integ_equiv} is trivially satisfied since $\iota_X\Omega = 0$. For condition (b), since $\d(e^0\wedge\Omega_0)=0$, we have
    $$ \d\Omega = -e^0 \wedge \Psi(\Omega_{ab}) = -\tr\left(A^T|_W\right)e^0 \wedge \Omega_{ab}. $$
    Since $e^0 \wedge e^0 = 0$, the right-hand side is equal to $-\tr\left(A^T|_W\right)e^0 \wedge (e^0 \wedge \Omega_0 + \Omega_{ab}) = \xi \wedge \Omega$, satisfying the condition as $\xi + \iota_X\beta = \xi$. Finally, item (b) gives $\d\beta \wedge \Omega = -e^0 \wedge (\Psi(\beta_{ab}) \wedge \Omega_{ab}) = 0$, satisfying condition (c). 
\end{proof}

\medskip

We finally prove Theorem \ref{thm:os-meninos-vem-como}:

\medskip

\begin{proof}[Proof of Theorem \ref{thm:os-meninos-vem-como}]
    First, let us analyze the case when, for $k \le n-1$,
    $$\mathrm{spec}(A)=\{\lambda_1, \lambda_1, \ldots, \lambda_k, \lambda_k, \mu, \zeta_1, -\zeta_1,
    \ldots, \zeta_{n-k-1}, -\zeta_{n-k-1}\}.$$

    Without loss of generality, we assign the basis elements such that:
    \begin{enumerate}[label=$\cdot$]
        \item $\d e^j = -\lambda_{\frac{j+1}{2}} e^0 \wedge e^j$ for $j \le 2k$ odd,
        \item $\d e^j = -\lambda_{\frac{j}{2}} e^0 \wedge e^j$ for $j \le 2k$ even,
        \item $\d e^{2k+1} = -\mu e^0 \wedge e^{2k+1}$,
        \item $\d e^j = -\zeta_{\frac{j-2k}{2}} e^0 \wedge e^j$ for $2k+2 \le j \le 2n-1$ even,
        \item $\d e^j = \zeta_{\frac{j-2k-1}{2}} e^0 \wedge e^j$ for $2k+2 \le j \le 2n-1$ odd.
    \end{enumerate}
    In other words, $e^{2k+2m}$ and $e^{2k+2m+1}$ carry the opposite eigenvalues $\zeta_m$ and
    $-\zeta_m$, for $m=1, \ldots, n-k-1$.

    Define the pure spinor $\rho = \exp({\mathbf i\omega}) \wedge \Omega$, where the complex $k$-form is
    given by $\Omega = \theta_1 \wedge \cdots \wedge \theta_k$, with
    $\theta_j = e^{2j-1} + \mathbf ie^{2j}$. Calculating the derivative of each complex generator:
    $$\d\theta_j = -\lambda_j e^0 \wedge \theta_j.$$
       Setting $\Lambda = \sum_{j=1}^{k} \lambda_j = \tr\left(A^T|_W\right)$, in accordance with Lemma
    \ref{lemma-integ-equiv_new}, the Leibniz rule gives
\[
\d\Omega=\sum_{j=1}^k(-1)^{j-1}
\theta_1\wedge\cdots\wedge\d\theta_j\wedge\cdots\wedge\theta_k
=-\left(\sum_{j=1}^k\lambda_j\right)e^0\wedge\Omega
=-\Lambda e^0\wedge\Omega.
\]
For $k=0$, the identity holds with $\Omega=1$ and $\Lambda=0$.

    Define the closed Darboux $2$-form on the complementary subspace:
    $$\omega = e^0 \wedge e^{2k+1} + \sum_{m=1}^{n-k-1} e^{2k+2m} \wedge e^{2k+2m+1}.$$ Note that $\d\omega = 0$: the first summand is closed because $\d e^0 = 0$ and
     $e^0 \wedge e^0 = 0$, whatever $\mu$ is, and the remaining ones because of the exact cancellation
    of the opposite eigenpairs.

    Calculating the total derivative of the spinor $\rho$, and using that $\exp(\mathbf i\omega)$ has
    even degree:
    $$\d\rho  = \mathbf i \exp(\mathbf i\omega) \wedge \d\omega \wedge \Omega - \Lambda e^0 \wedge
    \exp(\mathbf i\omega) \wedge \Omega = -\Lambda e^0 \wedge \rho.$$

    By selecting $X+\xi = 0 - \Lambda e^0 \in (\mathfrak{g}_A \oplus \mathfrak{g}_A^*)
    \otimes_{\mathbb R}\mathbb C$, we obtain the Clifford action
    $(X+\xi) \cdot \rho = \d\rho$, verifying integrability. Non-degeneracy is also satisfied
    since $\Omega \wedge \bar{\Omega}$ is a nonzero multiple of $e^1 \wedge \cdots \wedge e^{2k}$ while
    $\omega^{n-k}=(n-k)!\,e^0 \wedge e^{2k+1} \wedge \cdots \wedge e^{2n-1}$, so that
    $\omega^{n-k} \wedge \Omega \wedge \bar{\Omega} \neq 0$.

    \medskip

    Now, let us analyze the case when, for $k \ge 1$,
    $$\mathrm{spec}(A)=\{\lambda_1, \lambda_1, \ldots, \lambda_{k-1}, \lambda_{k-1}, \mu, \zeta_1,
    -\zeta_1, \ldots, \zeta_{n-k}, -\zeta_{n-k}\}.$$

    We use the analogous basis assignment: $e^{2j-1}$ and $e^{2j}$ carry the eigenvalue $\lambda_j$ for
    $j=1, \ldots, k-1$, the covector $e^{2k-1}$ carries $\mu$, and $e^{2k+2m-2}$, $e^{2k+2m-1}$ carry
    the opposite eigenvalues $\zeta_m$, $-\zeta_m$ for $m=1, \ldots, n-k$. Define the pure spinor
    $\rho = \exp(\mathbf i\omega) \wedge \Omega$, where the complex form of degree $k$ is given by
    $\Omega = \theta_1 \wedge \cdots \wedge \theta_k$, with $\theta_j = e^{2j-1} + \mathbf ie^{2j}$ for
    $j=1, \ldots, k-1$, and $\theta_k=e^0+\mathbf ie^{2k-1}$. In this case:
    $$\d\theta_j = -\lambda_j e^0 \wedge \theta_j\text{, }\forall j=1, \ldots, k-1;$$
    $$\d\theta_k=-\mu e^0 \wedge \theta_k,$$
    the latter because $\d e^0=0$ and $e^0 \wedge e^0 = 0$. Applying the Leibniz rule to $\Omega$, and
    defining $\Lambda := \mu + \sum_{j=1}^{k-1} \lambda_j$, we get:
    $$\d\Omega = -\Lambda e^0 \wedge \Omega.$$

    Define also the following closed Darboux $2$-form:
    $$\omega = \sum_{m=1}^{n-k} e^{2k+2m-2} \wedge e^{2k+2m-1}.$$ We get the following expression for the derivative of the pure spinor:
    $$\d\rho = -\Lambda e^0 \wedge \rho.$$

    By selecting $X+\xi = 0-\Lambda e^0$, we obtain $\d\rho =(X+\xi) \cdot \rho$, proving
    integrability. Non-degeneracy is also satisfied since $\Omega \wedge \bar{\Omega}$ is a nonzero
    multiple of $e^0 \wedge e^1 \wedge \cdots \wedge e^{2k-1}$ and
    $\omega^{n-k}=(n-k)!\,e^{2k} \wedge \cdots \wedge e^{2n-1}$, so that
    $\omega^{n-k} \wedge \Omega \wedge \bar{\Omega} \neq 0$.

    \medskip

       Conversely, take $A=\mathrm{diag}(\lambda_1, \lambda_2, \ldots, \lambda_{2n-1})$ and assume $G_A$
    admits a left-invariant generalized complex structure given by a pure spinor
    $\rho = \exp(\beta) \wedge \Omega$, where $\beta = B+\mathbf i\omega$, with $B$ and $\omega$ real
    $2$-forms, and $\Omega = \theta_1 \wedge \cdots \wedge \theta_k$, with $\theta_i=c_ie^0 + \eta_i$
    complex $1$-forms. Moreover, $\omega^{n-k} \wedge \Omega \wedge \overline{\Omega} \neq 0$ and, by
    Lemma \ref{lem:integ_equiv}, $\d\beta \wedge \Omega =0$.

    Let us first put both conditions in a convenient form. Lemma \ref{lemma-integ-equiv_new} implies the following. First, if one writes
    $\beta = \beta_{ab} + e^0 \wedge \beta_0$ and $\omega = \omega_{ab} + e^0 \wedge \omega_0$ with respect to the decomposition $\bigwedge^\bullet\mathfrak g_A^* = \bigwedge^\bullet \mathfrak h^* \oplus\, e^0
    \wedge \bigwedge^\bullet \mathfrak h^*$, then
    \begin{equation}\label{eq:integ-Psi}
    \Psi(\beta_{ab})\wedge\Omega_{ab}=0. 
    \end{equation}
   Moreover, $W=\mathrm{span}_{\mathbb C}\{\eta_1, \ldots, \eta_k\}$ is invariant
    under $A^T$. 
    
    Since $A^T$ is diagonal, $A^T|_W$ is diagonalizable and its eigenvalues are among the
    eigenvalues of $A$; let $v_i = x_i+\mathbf i y_i$ be an eigenbasis of $A^T|_W$, chosen adapted to
    the invariant subspace $W \cap \overline{W}$, with $A^Tv_i=\lambda_{j_i}v_i$ and hence, $A^T$ and
    $\lambda_{j_i}$ being real, $A^Tx_i=\lambda_{j_i}x_i$ and $A^Ty_i=\lambda_{j_i}y_i$. As
    $W + \overline{W}$ is stable under conjugation, there is a real subspace $V \subset \mathfrak{h}^*$
    with $V\otimes_{\mathbb R}\mathbb C = W + \overline{W}$, spanned by the $x_i$ and the $y_i$ and
    hence $A^T$-invariant. Since $V$ is spanned by eigenvectors, by exchange inside each eigenspace
    $E_\lambda=\mathrm{span}\{e^p : \lambda_p=\lambda\}$ we may complete a basis of $V$ to a real
    eigenbasis $\{f^1, \ldots, f^{2n-1}\}$ of $\mathfrak{h}^*$, and we let $V'$ be spanned by the
    remaining ones, so that $\mathfrak{h}^*=V\oplus V'$. We index the first family by $J$ and the
    second by $I$, and work in this basis from now on, which is legitimate because
    \eqref{eq:integ-Psi} holds in any real eigenbasis.

        Everything now follows from two observations. Writing
        $\bigwedge^\bullet\left(\mathfrak h^*\otimes_{\mathbb R}\mathbb C\right)=\bigoplus_{a,b}
       \bigwedge^a\left(V\otimes_{\mathbb R}\mathbb C\right)\wedge
        \bigwedge^b\left(V'\otimes_{\mathbb R}\mathbb C\right)$, we decompose
            $\beta_{ab}=\beta_{VV}+\beta_{VV'}+\beta_{V'V'}$ and
      $\omega_{ab}=\omega_{VV}+\omega_{VV'}+\omega_{V'V'}$.

    \emph{Observation 1.} As $A^T$ preserves both $V$ and $V'$, the derivation $\Psi$ preserves this
    bidegree decomposition, while $\Omega_{ab}\in\bigwedge^k(V\otimes_{\mathbb R}\mathbb C)$ has
    bidegree $(k,0)$. Hence, each bidegree component of \eqref{eq:integ-Psi} vanishes separately; the
    one of bidegree $(k,2)$ gives $\Psi(\beta_{V'V'})\wedge\Omega_{ab}=0$ and, $\Omega_{ab}$ being
    nonzero with $V\cap V'=\{0\}$, we get $\Psi(\beta_{V'V'})=0$, that is,
    \begin{equation}\label{eq:opposite}
    \beta_{pq}\left(\lambda_p+\lambda_q\right)=0 \quad\text{for all } p,q \in I .
    \end{equation}
    In particular, $\beta_{pq}=B_{pq}+\mathbf i\,\omega_{pq}$ with $B_{pq}$ and $\omega_{pq}$ real, so
    that $\omega_{pq}\neq0$ implies $\beta_{pq}\neq0$ and therefore $\lambda_p=-\lambda_q$.

    \emph{Observation 2.} If $\omega_{V'V'}^{\,m}\neq0$, then, expanding it in the basis, some monomial
    has a nonzero coefficient, so there are $2m$ distinct indices in $I$ forming disjoint pairs
    $(p_1,q_1), \ldots, (p_m,q_m)$ with $\omega_{p_iq_i}\neq0$. By \eqref{eq:opposite}, these are $m$
    disjoint pairs of opposite eigenvalues.

    \medskip

    \textbf{Case 1:} $W \cap \overline{W}=\{0\}$

    Here $V\otimes_{\mathbb R}\mathbb C=W\oplus\overline{W}$ has dimension $2k$, so
    $x_1 \wedge y_1 \wedge \cdots \wedge x_k \wedge y_k \neq 0$ and
    $\dim_{\mathbb R}V'=2(n-k)-1$. For each eigenvalue $\lambda$ of $A$, the vectors $x_i,y_i$ with
    $\lambda_{j_i}=\lambda$ are $2m_\lambda$ linearly independent elements of $E_\lambda$, whence
    $\dim E_\lambda \ge 2m_\lambda$; extracting $m_\lambda$ disjoint identical pairs for each $\lambda$
    and summing, $V$ contributes $k$ disjoint pairs of identical eigenvalues.

    We claim that $\omega_{V'V'}^{\,n-k-1}\neq0$. Indeed, since
    $\Omega \wedge \overline{\Omega}=e^0\wedge\Omega'+\Omega_{ab}\wedge\overline{\Omega_{ab}}$ with
    $\Omega'=\Omega_0 \wedge\overline{\Omega_{ab}}+ (-1)^{k} \Omega_{ab} \wedge\overline{\Omega_0}$,
    and since $\omega_{ab}^{\,n-k}\wedge\Omega_{ab}\wedge\overline{\Omega_{ab}}$ vanishes for being a
    $2n$-form on the $(2n-1)$-dimensional space $\mathfrak h^*$, we get
      $$0 \neq \omega^{n-k} \wedge \Omega \wedge \overline{\Omega}= e^0 \wedge \left(
    \omega_{ab}^{\,n-k} \wedge \Omega' + (n-k)\, \omega_0 \wedge \omega_{ab}^{\,n-k-1} \wedge
    \Omega_{ab} \wedge \overline{\Omega_{ab}} \right).$$
    Suppose the first term is nonzero. As $\Omega'$ has bidegree $(2k-1,0)$ and $\dim_{\mathbb R}V=2k$,
    only the bidegree $(1,2(n-k)-1)$ component of $\omega_{ab}^{\,n-k}$ can survive against it, and
    that component is $(n-k)\,\omega_{VV'}\wedge\omega_{V'V'}^{\,n-k-1}$; in particular
    $\omega_{V'V'}^{\,n-k-1}\neq0$. Suppose instead the second term is nonzero. Now
    $\Omega_{ab}\wedge\overline{\Omega_{ab}}$ has bidegree $(2k,0)$, so
    $\omega_0 \wedge \omega_{ab}^{\,n-k-1}$ must contribute in bidegree $(0,2(n-k)-1)$, forcing again
    $\omega_{V'V'}^{\,n-k-1}\neq0$. This proves the claim.

    By Observation 2, $V'$ contains $n-k-1$ disjoint pairs of opposite eigenvalues. These use
    $2(n-k)-2$ of the $2(n-k)-1$ eigenvalues indexed by $I$, and are disjoint from the $k$ identical
    pairs found in $V$; exactly one eigenvalue is left over, which we call $\mu$. Therefore the
    spectrum of $A$ takes the first form of the statement.

    \medskip

    \textbf{Case 2:} $W \cap \overline{W}\neq\{0\}$

    We first show that $\dim_{\mathbb C}\left(W \cap \overline{W}\right)=1$. This subspace is stable
    under conjugation, and, since $W$ is $A^T$-invariant with $A^T$ real, it is $A^T$-invariant as well.
    Hence it admits a basis of real eigenvectors with our eigenbasis chosen adapted to it. Suppose,
    for the sake of contradiction, that two of the $v_i$ are real, say $v_1$ and $v_2$. Then every
    $\Omega_{\hat\jmath}$ contains $v_1$ or $v_2$, both of which occur in $\overline{\Omega_{ab}}$, so
    that, $\Omega_0$ being a combination of the $\Omega_{\hat\jmath}$,
    $$\Omega_0 \wedge \overline{\Omega_{ab}}=\Omega_{ab} \wedge \overline{\Omega_0}=0,$$
    whence $\Omega'=0$. But $\Omega_{ab}\wedge\overline{\Omega_{ab}}=0$ in this case, as any nonzero
    vector of $W\cap\overline W$ occurs in both factors, so that
    $\Omega \wedge \overline{\Omega}=e^0\wedge\Omega'=0$, a contradiction.

    So $v_1$ is real and $v_2, \ldots, v_k$ are not. Setting
    $W'=\mathrm{span}_{\mathbb C}\{v_2, \ldots, v_k\}$, we get
    $W' \cap \overline{W'} \subseteq W\cap\overline W=\mathbb{C}v_1$ with $v_1 \notin W'$, hence
    $W' \cap \overline{W'}=\{0\}$ and $x_2 \wedge y_2 \wedge \cdots \wedge x_k \wedge y_k \neq 0$.
    Arguing as in Case 1 with $W'$ in place of $W$, we obtain exactly $k-1$ disjoint pairs of identical
    eigenvalues, together with the single eigenvalue $\mu:=\lambda_{j_1}$ carried by $v_1$. Here
    $\dim_{\mathbb R}V=\dim_{\mathbb C}\left(W+\overline W\right)=2k-\dim_{\mathbb C}
    \left(W\cap\overline W\right)=2k-1$ and $\dim_{\mathbb R}V'=2(n-k)$, and
    $\Omega \wedge \overline{\Omega}=e^0\wedge\Omega'\neq0$ with $\Omega'$ of bidegree $(2k-1,0)$,
    hence a generator of the line $\bigwedge^{2k-1}\left(V\otimes_{\mathbb R}\mathbb C\right)$. Since
    $e^0 \wedge e^0 = 0$, only $\omega_{ab}^{\,n-k}$ survives in
     $$0 \neq \omega^{n-k}\wedge e^0 \wedge \Omega' = e^0 \wedge \omega_{ab}^{\,n-k}\wedge\Omega',$$ 
    and, $\Omega'$ being of top degree in $V$, only the bidegree $(0,2(n-k))$ component of
    $\omega_{ab}^{\,n-k}$ can survive, that is, $\omega_{V'V'}^{\,n-k}\neq0$. By Observation 2, $V'$
    splits into $n-k$ disjoint pairs of opposite eigenvalues, which exhaust it. Therefore the spectrum
    of $A$ consists of $k-1$ identical pairs, the eigenvalue $\mu$, and $n-k$ opposite pairs, taking
    the second form of the statement.
\end{proof}

\medskip

Next, we prove Corollary \ref{cor-CY-diag}. 

\medskip

\subsection{Left-invariant generalized Calabi--Yau structures}
Now we move on to discuss the left-invariant generalized complex structures for which the pure
spinor $\rho$ is closed, that is, the left-invariant generalized Calabi--Yau structures.

\begin{lemma} \label{lemma_Psi}
    Let $\rho = \exp(\beta) \wedge \Omega$ be a pure spinor, and write
    $\beta=e^0\wedge\beta_0+\beta_{ab}$ and $\Omega = e^0 \wedge \Omega_0 + \Omega_{ab}$. Then $\rho$
    is closed if and only if
    $$\Psi(\Omega_{ab}) = 0 \quad\text{and}\quad \Psi(\beta_{ab}) \wedge \Omega_{ab} = 0 .$$
\end{lemma}
\begin{proof}
    Since $\d=-e^0\wedge\Psi(\cdot)$ and $\Psi$ is a derivation,
    $$\d\rho = -e^0 \wedge \exp(\beta) \wedge \left(\Psi(\beta) \wedge \Omega + \Psi(\Omega)\right).$$
    Now $\Psi(e^0\wedge\gamma)=e^0\wedge\Psi(\gamma)$, so that
    $\Psi(\beta)=e^0\wedge\Psi(\beta_0)+\Psi(\beta_{ab})$ and
    $\Psi(\Omega) = e^0 \wedge \Psi(\Omega_0) + \Psi(\Omega_{ab})$. As every term carrying an $e^0$ is
    annihilated by the outer $e^0 \wedge (\cdot)$, we get
    $$\d\rho = -e^0 \wedge \exp(\beta) \wedge \left(\Psi(\beta_{ab}) \wedge \Omega_{ab}
    + \Psi(\Omega_{ab})\right).$$
    Since $\exp(\beta)\wedge(\cdot)$ is invertible, with inverse $\exp(-\beta)\wedge(\cdot)$, and
    $e^0 \wedge (\cdot)$ is injective on $\bigwedge^\bullet \mathfrak{h}^*
    \otimes_{\mathbb R}\mathbb{C}$, the condition $\d\rho=0$ is equivalent to
    $$\Psi(\beta_{ab}) \wedge \Omega_{ab} + \Psi(\Omega_{ab})=0 .$$
    Finally, $\Psi$ preserves the degree, so the two summands above have degrees $k+2$ and $k$
    respectively, and therefore vanish separately.
\end{proof}
Applying the integrability constraints of Lemma \ref{lemma-integ-equiv_new} to the closure conditions of Lemma \ref{lemma_Psi} reduces the Calabi--Yau requirement to a single algebraic condition, as follows.

\begin{theo}\label{thm-Calabi-Yau}
Let $\rho=\exp\beta\wedge\Omega$ be an invariant non-degenerate integrable pure spinor, with associated invariant subspace $W$.
Set $\Lambda=\operatorname{tr}(A^T|_W)$. Then
\[
\d\rho=-\Lambda e^0\wedge\rho,\qquad
\d\rho=0\ \Longleftrightarrow\ \Lambda=0.
\]
\end{theo}
\begin{proof}
Lemma \ref{lemma-integ-equiv_new} gives $\Psi\Omega_{ab}=\Lambda\Omega_{ab}$ and
$\Psi(\beta_{ab})\wedge\Omega_{ab}=0$. Substitution in the calculation of Lemma \ref{lemma_Psi} proves the formula.
Also $e^0\wedge\rho\ne0$, since multiplication by $\exp\beta$ is invertible and $\Omega_{ab}\ne0$. Other invariant generators
differ by nonzero constants.
\end{proof}

\medskip

\begin{proof}[Proof of Corollary \ref{cor-CY-diag}]
    Since a generalized Calabi--Yau structure of type $k$ is, in particular, a generalized complex
    structure of type $k$, Theorem \ref{thm:os-meninos-vem-como} says that $\mathrm{spec}(A)$ assumes
    one of the two forms in the statement, according to the two cases of its proof.

    In Case 1, the eigenbasis $v_1, \ldots, v_k$ of $A^T|_W$ carries exactly the eigenvalues
    $\lambda_1, \ldots, \lambda_k$ of the $k$ identical pairs. Therefore, by Theorem
    \ref{thm-Calabi-Yau}, $$0=\tr\left(A^T|_W\right)=\sum_{i=1}^k \lambda_i.$$

    In Case 2, the eigenbasis of $A^T|_W$ carries the eigenvalues $\lambda_1, \ldots, \lambda_{k-1}$ of
    the identical pairs, together with the eigenvalue $\mu$ of the real eigenvector $v_1$. Hence,
    Theorem \ref{thm-Calabi-Yau} reads
    $$0=\tr\left(A^T|_W\right)= \mu +\sum_{i=1}^{k-1} \lambda_i.$$

    The converse follows from the proof of Theorem \ref{thm:os-meninos-vem-como}: the structures
    constructed there satisfy $\d\rho = -\Lambda e^0 \wedge \rho$, with
    $\Lambda=\sum_{i=1}^k \lambda_i$ in the first form and $\Lambda=\mu+\sum_{i=1}^{k-1} \lambda_i$ in
    the second. Under the additional hypotheses of the statement, we have $\Lambda=0$ in either case,
    so that $\d\rho=0$.
\end{proof}

\medskip

This concludes the discussion of the diagonal case. We now turn to left-invariant generalized complex
structures on more general almost abelian Lie groups, namely those determined by non-diagonal
matrices.

\medskip

\subsection{Generalized complex structures via algebraic reduction} Beyond the diagonal case, we establish a constructive method for left-invariant generalized complex structures on $\mathfrak{g}_A$ via algebraic reduction. Analogous to the construction presented in \cite{ACK}*{Section 2}, our approach couples a complex structure on a quotient algebra with a symplectic form on an $A$-invariant ideal.

\begin{theo} \label{thm:gcs_from_fibration}
Let $G_A$ be a $2n$-dimensional almost abelian Lie group with Lie algebra
$\mathfrak{g}_A = \mathbb{R}e_0 \ltimes_A \mathfrak{h}$. Suppose that $\mathfrak{h}$ decomposes into
$A$-invariant subspaces $\mathfrak{h} = \mathfrak{h}_c \oplus \mathfrak{h}_s$, with
$\dim \mathfrak{h}_c = 2k-1$ and $\dim \mathfrak{h}_s = 2(n-k)$, where $1\le k\le n$, and write $A_c=A|_{\mathfrak h_c}$
and $A_s=A|_{\mathfrak h_s}$. Assume that:
\begin{enumerate}
    \item the quotient Lie algebra $\mathfrak{g}_c := \mathfrak{g}_A/\mathfrak{h}_s \cong
    \mathbb{R}e_0 \ltimes_{A_c} \mathfrak{h}_c$ admits a complex structure;
    \item there exists a closed $2$-form $\omega \in \bigwedge^2 \mathfrak{g}_A^*$ restricting to a
    non-degenerate form $\sigma \in \bigwedge^2 \mathfrak{h}_s^*$ on the ideal $\mathfrak{h}_s$.
\end{enumerate}
Then $G_A$ admits a left-invariant generalized complex structure of type $k$.
\end{theo}

\begin{proof}
Let $\Omega \in \bigwedge^k \mathfrak{g}_c^* \otimes \mathbb{C}$ be a non-degenerate pure spinor
generating the complex structure on $\mathfrak{g}_c$, which has dimension $2k$. By integrability there
is $\alpha \in \mathfrak{g}_c^* \otimes \mathbb{C}$ with $\d_c\Omega = \alpha \wedge \Omega$, where
$\d_c$ is the differential of $\mathfrak{g}_c$. Let $\pi\colon\mathfrak{g}_A \to \mathfrak{g}_c$ be
the canonical projection and $\iota\colon\mathfrak{h}_s \hookrightarrow \mathfrak{g}_A$ the inclusion,
so that $\iota^*\omega = \sigma$. Define
\[ \rho = \exp({\mathbf i\omega}) \wedge \pi^*\Omega \in \textstyle\bigwedge^\bullet \mathfrak{g}_A^*
\otimes \mathbb{C}. \]
Since $\Omega$ is decomposable, so is $\pi^*\Omega$, and $\rho$ is a pure spinor of type $k$.

For non-degeneracy, $\Omega \wedge \bar{\Omega} \neq 0$ gives
$\pi^*\Omega \wedge \overline{\pi^*\Omega} = \pi^*\left(\Omega \wedge \bar{\Omega}\right) \neq 0$, a
generator of the line $\bigwedge^{2k}\pi^*\mathfrak g_c^*$, while $\sigma^{n-k}\neq0$ generates
$\bigwedge^{2(n-k)}\mathfrak h_s^*$. Now decompose
$\omega=\omega_{ss}+\omega_{cs}+\omega_{cc}$ according to
$\mathfrak{g}_A^* \cong \pi^*\mathfrak{g}_c^* \oplus \mathfrak{h}_s^*$, with
$\omega_{ss}=\sigma$. In the product $\omega^{n-k}\wedge\pi^*\Omega\wedge\overline{\pi^*\Omega}$, the
second factor already occupies all $2k$ directions of $\pi^*\mathfrak g_c^*$, so every monomial of
$\omega^{n-k}$ involving $\omega_{cs}$ or $\omega_{cc}$ is annihilated, and only $\sigma^{n-k}$
survives:
\[ \omega^{n-k} \wedge \pi^*\Omega \wedge \overline{\pi^*\Omega}
= \sigma^{n-k} \wedge \pi^*\Omega \wedge \overline{\pi^*\Omega} \neq 0, \]
the two factors lying in complementary subspaces of $\bigwedge^\bullet\mathfrak g_A^*$ of
complementary degrees $2(n-k)$ and $2k$. Thus $\rho$ is non-degenerate.

For integrability, since $\d\omega=0$ and $\d\circ\pi^*=\pi^*\circ\d_c$,
\[ \d\rho = \mathbf i\,\d\omega\wedge \exp({\mathbf i\omega})\wedge\pi^*\Omega
+ \exp({\mathbf i\omega}) \wedge \pi^*\left(\d_c\Omega\right)
= \exp({\mathbf i\omega}) \wedge \pi^*\alpha \wedge \pi^*\Omega = \pi^*\alpha \wedge \rho, \]
so that $\d\rho=\left(0+\pi^*\alpha\right)\cdot\rho$ and Lemma \ref{lem:integ_equiv} applies with
$X+\xi=\pi^*\alpha$. Hence $\rho$ defines a left-invariant generalized complex structure of type $k$
since $\pi^*\Omega$ has degree $k$.
\end{proof}

Condition (i) can be checked using \cite{Arroyo2025}. For condition (ii), combine the symplectic Jordan criterion of \cite{Arroyo2025}*{Theorem 5.6} with extension by zero along
the specified invariant complement.

\begin{prop} \label{prop:jordan_equivalence}
Let $\mathfrak g_A=\mathbb Re_0\ltimes_A\mathfrak h$ and assume $\mathfrak h=\mathfrak h_c\oplus\mathfrak h_s$ is a decomposition
into $A$-invariant real subspaces, with $\dim\mathfrak h_s=2(n-k)$. Put $A_s=A|_{\mathfrak h_s}$. In a Jordan-basis formulation, it suffices that $\mathfrak h_s$
be the sum of entire real Jordan blocks. Then the following are equivalent:
\begin{enumerate}[label=(\roman*)]
    \item condition (ii) of Theorem \ref{thm:gcs_from_fibration} holds for $\mathfrak h_s$;
    \item $A_s\in\mathfrak{sp}\left(\mathfrak h_s,\sigma\right)$ for some symplectic form $\sigma$ on
    $\mathfrak h_s$;
    \item the real Jordan form of $A_s$ satisfies:
    \begin{enumerate}[label=(\alph*)]
        \item the blocks ${\mathcal J}_m(\lambda)$ with $\lambda\in\mathbb R\setminus\{0\}$ occur in matching
        pairs ${\mathcal J}_m(\lambda)$, ${\mathcal J}_m(-\lambda)$;
        \item the blocks $C_m(a,b)$ with $a\neq0$ and $b\neq0$ occur in matching pairs $C_m(a,b)$,
        $C_m(-a,b)$;
        \item the nilpotent blocks ${\mathcal J}_m(0)$ of odd size $m$ occur with even multiplicity.
    \end{enumerate}
    No condition is imposed on the blocks $C_m(0,b)$ with $b\neq0$.
\end{enumerate}

\end{prop}
\begin{proof}
\emph{(i) $\Rightarrow$ (ii).} Let $\omega = e^0 \wedge \omega_0 + \omega_{ab} \in \bigwedge^2
\mathfrak{g}_A^*$ be closed with $\iota^*\omega = \sigma$ non-degenerate on $\mathfrak{h}_s$. By
Proposition \ref{prop:ext-derivative}, $\d\omega = - e^0 \wedge \Psi_A(\omega_{ab})$, so $\d\omega=0$
gives $\Psi_A(\omega_{ab}) = 0$. Restricting to the $A$-invariant subspace $\mathfrak h_s$, where
$\omega_{ab}$ restricts to $\sigma$, yields $\Psi_{A_s}(\sigma)=0$, that is,
\begin{equation}\label{eq:sp-condition}
\sigma(A_s u, v) + \sigma(u, A_s v)=0,\qquad u,v\in\mathfrak h_s.
\end{equation}

\emph{(ii) $\Longleftrightarrow$ (iii).}
A non-degenerate alternating matrix $S$ represents an invariant
symplectic form when $A_s^TS+SA_s=0$.
Equivalently, $A_s$ is conjugate to an element of the standard
symplectic Lie algebra. Theorem 5.6 of \cite{Arroyo2025}
gives precisely the conditions in (iii).
For paired real blocks, put $P_{ij}=(-1)^{i-1}\delta_{i+j,m+1}$.
Then
\[
J_m(\lambda)^TP+PJ_m(-\lambda)=0,
\]
and $\left(\begin{smallmatrix}0&P\\-P^T&0\end{smallmatrix}\right)$
is a preserved non-degenerate alternating matrix on their direct sum.
For one nilpotent block of even size, $P$ itself is alternating.
The remaining nilpotent and non-real cases are included in the
cited theorem.

\emph{(ii) $\Rightarrow$ (i).} Let $\sigma$ be as in (ii), so that $\Psi_{A_s}(\sigma)=0$ by
\eqref{eq:sp-condition}, and let $p_s\colon\mathfrak{g}_A \to \mathfrak{h}_s$ be the projection
determined by the $A$-invariant decomposition
$\mathfrak g_A=\left(\mathbb Re_0\oplus\mathfrak h_c\right)\oplus\mathfrak h_s$, where
$\mathfrak h_c$ is the invariant complement in the hypothesis. Set $\omega := p_s^*\sigma$, which
restricts to $\sigma$ on $\mathfrak h_s$ by construction. Since the decomposition is $A$-invariant,
$\Psi_A \circ p_s^* = p_s^* \circ \Psi_{A_s}$.  Proposition \ref{prop:ext-derivative} then gives
\[ \d\omega = -e^0 \wedge \Psi_A\left(p_s^*\sigma\right)
= -e^0 \wedge p_s^*\left(\Psi_{A_s}\sigma\right)= 0 . \]
Thus $\omega$ is a closed $2$-form on $\mathfrak g_A$ restricting to a non-degenerate form on
$\mathfrak h_s$, verifying Theorem \ref{thm:gcs_from_fibration}(ii).
\end{proof}

Theorem \ref{thm:gcs_from_fibration} and Proposition \ref{prop:jordan_equivalence} together give a
direct algebraic recipe for producing almost abelian Lie groups with left-invariant generalized
complex structures, as we register in the next remark.

\begin{remark}\label{rem:gcs_algorithm}\textbf{A constructive algorithm.}
Given $A$ in real Jordan form, one looks for a block partition
$A = A_c \oplus A_s$ such that
\begin{enumerate}
    \item $A_c$ determines a $2k$-dimensional almost abelian Lie algebra admitting a left-invariant complex
    structure, and
    \item $A_s$ has its blocks paired as in Proposition \ref{prop:jordan_equivalence}.
\end{enumerate}
If such a partition exists, then $G_A$ admits a left-invariant generalized complex structure of type
$k$.
\end{remark}

\medskip

\subsection{Type obstructions in the non-diagonal case}

In this subsection we prove the following:

\begin{theo} \label{thm:type-obstruction}
Let $G_A$ be an almost abelian Lie group with Lie algebra $\mathfrak{g}_{A}=\mathbb{R}e_{0}\ltimes_{A}\mathfrak{h}$. 
For each real eigenvalue $\lambda$ of the matrix $A$, let $s_\lambda$ denote its geometric multiplicity, and suppose its associated generalized eigenspace decomposes into real Jordan blocks of sizes $m_1^{(\lambda)} \ge m_2^{(\lambda)} \ge \cdots \ge m_{s_\lambda}^{(\lambda)}$. Let $m_\lambda = \sum_{i=1}^{s_\lambda} m_i^{(\lambda)}$ be the algebraic multiplicity of $\lambda$. If $G_A$ admits a left-invariant generalized complex structure of type $k$, then:
$$k \leq \min_{\lambda \in \mathrm{spec}(A)\cap\mathbb R} \left( \sum_{j=1}^{\lfloor s_\lambda/2 \rfloor} m_{2j}^{(\lambda)} + \left\lfloor \frac{2n - 1 - m_\lambda}{2} \right\rfloor +1 \right)$$
\end{theo}

We start with an auxiliary result. For context, recall that Lemma \ref{lemma-integ-equiv_new} states that the
existence of a left-invariant generalized complex structure of type $k$ requires a $k$-dimensional
$A^T$-invariant complex subspace $W \subset \mathfrak{h}^* \otimes_{\mathbb R} \mathbb{C}$. The
following strengthens this condition:

 \begin{lemma} \label{lem:dim-W-conj}
     Let $G_A$ admit a left-invariant generalized complex structure of type $k \ge 1$, with
     non-degenerate pure spinor $\rho = \exp(\beta) \wedge \Omega$, and let $W$ be as in Lemma
     \ref{lemma-integ-equiv_new}. Then
     $$2k-1 \le \dim_{\mathbb{C}}\left(W+\overline{W}\right).$$
 \end{lemma}
 \begin{proof}
     Expanding $\Omega \wedge \overline{\Omega}$ with respect to the decomposition
     $\Omega = e^0\wedge \Omega_0 + \Omega_{ab}$ yields, as before,
     $$ \Omega\wedge\overline{\Omega} = e^{0}\wedge\Omega' + \Omega_{ab}\wedge\overline{\Omega_{ab}},
     \qquad \Omega'=\Omega_{0}\wedge\overline{\Omega_{ab}} +
     (-1)^{k}\Omega_{ab}\wedge\overline{\Omega_{0}} .$$
     Since $\Omega_0 \in \bigwedge^{k-1}W$ and $\Omega_{ab}$ spans $\bigwedge^{k}W$, both summands of
     $\Omega'$ lie in $\bigwedge^{2k-1}\left(W+\overline{W}\right)$, while
     $\Omega_{ab}\wedge \overline{\Omega_{ab}}$ lies in
     $\bigwedge^{2k}\left(W+\overline{W}\right)$. By non-degeneracy $\Omega \wedge \overline{\Omega}
     \neq 0$, so at least one of these components is nonzero, forcing
     $\bigwedge^{j}\left(W + \overline{W}\right) \neq 0$ for $j = 2k-1$ or $j=2k$. Either way,
     $2k-1 \le \dim_{\mathbb{C}}\left(W+\overline{W}\right)$.
 \end{proof}

\begin{prop} \label{prop:k-leq-1}
    Let $G_A$ be an almost abelian Lie group with Lie algebra $\mathfrak{g}_A = \mathbb{R} e_0 \ltimes_A \mathfrak{h}$. Assume $ \mathrm{spec}(A) \subset \mathbb{R}$ and every $\lambda \in  \mathrm{spec}(A) $ has geometric multiplicity $1$. If $G_A$ admits a left-invariant generalized complex structure of type $k$, then  $k \leq 1$.
\end{prop}
\begin{proof}
    The case $k=0$ is trivial; assume $k \ge 1$. Since $\mathrm{spec}(A) \subset \mathbb{R}$ and
    every eigenvalue has geometric multiplicity $1$, each generalized eigenspace of $A^T$ consists of
    a single Jordan block, and this remains true after complexification. Fix
    such a block $J_{m_i}(\lambda_i)$, with adapted basis $e^1, \ldots, e^{m_i}$, and let
    $N:=A^T-\lambda_i$ on it, so that $N$ is nilpotent with $\dim_{\mathbb C}\ker N = 1$. If
    $U \neq \{0\}$ is $A^T$-invariant, then it is $N$-invariant, and taking $0 \neq u \in U$ with
    $N^{d}u=0 \neq N^{d-1}u$ we get $\ker N \subseteq U$; iterating, $U=\ker N^{\dim_{\mathbb C}U}$.
    Hence the only $A^T$-invariant subspaces inside the block are
    $$ W_{m_i} = \mathrm{span}_{\mathbb{C}}\{e^{m_i-j+1}, \dots, e^{m_i}\},\qquad 1 \le j \le m_i,$$
    which are spanned by real vectors and therefore satisfy $W_{m_i} = \overline{W_{m_i}}$. Every
    $A^T$-invariant subspace $W$ decomposes as
    $W = W_{\lambda_1} \oplus W_{\lambda_2} \oplus \cdots \oplus W_{\lambda_r}$, with each
    $W_{\lambda_i}$ invariant inside the block of $\lambda_i$; consequently $W = \overline{W}$ and
    $\dim_{\mathbb C}\left(W+\overline W\right)=k$. By Lemma \ref{lem:dim-W-conj},
    $$2k-1 \le \dim_{\mathbb{C}}\left(W+\overline{W}\right)=k \implies k\leq 1. $$
\end{proof}
In general, invariant subspaces need not intersect their conjugates trivially, and a coarser estimate
is required. We now establish such a bound for arbitrary Jordan type. We note that Proposition
\ref{prop:k-leq-1} is not a particular case of Theorem \ref{thm:type-obstruction}: under its
hypotheses, the bound of the theorem reads
$k \le \min_\lambda \left(\lfloor (2n-1-m_\lambda)/2 \rfloor + 1\right)$, which is weaker than
$k \le 1$.

\begin{proof}[Proof of Theorem \ref{thm:type-obstruction}]

By Lemma \ref{lem:dim-W-conj}, 
\begin{align} \label{deg}
    2k-1 \le \dim_{\mathbb{C}}(W+\overline{W}).
\end{align}
Fixing a real eigenvalue $\lambda \in \mathrm{spec}(A)$, let $V_\lambda \subset \mathfrak{h}^* \otimes \mathbb{C}$ be its generalized complexified eigenspace and $V^c$ be the sum of all remaining generalized eigenspaces, also complexified, which has dimension $d = 2n - 1 - m_\lambda$. The invariance of $W$ yields
$$W = W_\lambda \oplus W_r,$$ 
where $W_\lambda = W \cap V_\lambda$ and $W_r = W \cap V^c$. Let $k_\lambda = \dim W_\lambda$ and $k_r = \dim W_r$, so that $k = k_\lambda + k_r$. Since the subspaces $V_\lambda$ and
$V^c$ are invariant under complex conjugation,
$$2k-1 \leq \dim(W + \overline{W}) = \dim(W_\lambda + \overline{W}_\lambda) + \dim(W_r + \overline{W}_r).$$We now bound each component independently:
\begin{align} \label{eq:dim_Wr}
    \dim(W_r + \overline{W_r}) = 2k_r - \dim(W_r \cap \overline{W_r}).
\end{align}

Setting $c := k_r - \dim(W_r \cap \overline{W}_r)$, we may rewrite \eqref{eq:dim_Wr} as $\dim(W_r + \overline{W}_r) = k_r + c$. Since $\mathrm{dim} (W_r + \overline{W}_r)$ cannot exceed the total dimension $d$, and $c \leq k_r$, it follows from $2c \leq k_r + c \leq d $ that$$ c \leq \left\lfloor \frac{d}{2} \right\rfloor.$$
Then, 
\begin{equation} \label{bound1}
    \dim(W_r + \overline{W_r}) \leq k_r + \left\lfloor \frac{2n-1-m_{\lambda}}{2} \right\rfloor.
\end{equation}

To establish an upper bound for $\dim(W_\lambda + \overline{W}_\lambda)$, we explicitly construct the $A^T$-invariant subspace $\mathcal{W}_\lambda$ as follows.
Let $\{e^{r,1}, \dots, e^{r,m_r}\}$ denote the real Jordan basis for the $r$-th block ($1 \le r \le s_\lambda$), satisfying 
\begin{align}
    A^T e^{r,i} =& \lambda e^{r,i} + e^{r,i+1} \quad \text{for} \quad i \in \{1, \dots, m_r -1\}\\
    A^T e^{r, m_r} &= \lambda e^{r, m_r}
\end{align}
For each pair of blocks $(m_{2j-1}, m_{2j})$ with $1 \le j \le \lfloor s_\lambda/2 \rfloor$, set $p = m_{2j-1}$ and $q = m_{2j}$. Since $p \ge q$, we can define $q$ complex covectors:$$ \eta^{(j)}_t := e^{2j-1, \, p-q+t} + \mathbf{i}e^{2j, \, t}, \quad 1 \le t \le q. $$

By linearity, applying $A^T$ yields $A^T\eta_t^{(j)} = \lambda\eta_t^{(j)} + \eta_{t+1}^{(j)}$, with $\eta_{q+1}^{(j)} = 0$. As a result, $W_j := \text{span}_{\mathbb{C}}\{\eta_1^{(j)}, \dots, \eta_q^{(j)}\}$ is an $A^T$-invariant subspace of dimension $m_{2j}$. For each $1 \le j \le \lfloor s_\lambda/2 \rfloor$, it follows that $W_j \cap \overline{W}_j = \{0\}$.

We define $\mathcal{W}_\lambda := \bigoplus_{j=1}^{\lfloor s_\lambda/2 \rfloor} W_j$. The intersection remains trivial $\mathcal{W}_\lambda \cap \overline{\mathcal{W}}_\lambda = \{0\}$ and $\dim(\mathcal{W}_\lambda) = \sum_{j=1}^{\lfloor s_\lambda/2 \rfloor} m_{2j}$. Consequently, for this construction, we obtain$$\dim \mathcal{W}_\lambda - \dim(\mathcal{W}_\lambda \cap \overline{\mathcal{W}}_\lambda) = \sum_{j=1}^{\lfloor s_\lambda/2 \rfloor} m_{2j}.$$
In Appendix \ref{appendix:A} we prove that $\mathcal{W}_\lambda$ maximizes $\dim W_\lambda -\dim (W_\lambda \cap \overline{W}_\lambda)$ among all possible invariant subspaces.  
Therefore, for any $A^T$-invariant subspace $W_\lambda$, we have
\begin{align} \label{bound2}
    \dim({W}_\lambda +\overline{{W}_\lambda}) \leq {k}_\lambda + \sum_{j=1}^{\lfloor s_\lambda/2 \rfloor} m_{2j}
\end{align}

Equations \eqref{bound1} and \eqref{bound2} allow us to then write
  $$\dim(W + \overline{W}) \le k + \sum_{j=1}^{\lfloor s_{\lambda}/2 \rfloor} m_{2j}^{(\lambda)} + \left\lfloor\frac{2n-1-m_{\lambda}}{2}\right\rfloor$$
By \eqref{deg}, 
$$k \le \sum_{j=1}^{\left\lfloor s_{\lambda}/2 \right\rfloor} m_{2j}^{(\lambda)} + \left\lfloor\frac{2n-1-m_{\lambda}}{2}\right\rfloor + 1,$$
which holds for every $\lambda\in\mathrm{spec}(A)\cap\mathbb R$. 
\end{proof}

Moreover, it is immediate to observe that in the nilpotent case, where $\lambda=0$ is the only
eigenvalue and we drop it from the notation, $m_\lambda=2n-1$ and the bound reads
$$k \leq \sum_{j=1}^{\lfloor s/2 \rfloor} m_{2j} + 1.$$
The nilpotent case is also distinguished for the following reason, as first established in \cite{Cavalcanti2004}.

\begin{cor}\label{cor:nilpotent-CY}
    Let $G_A$ be an almost abelian Lie group with $A$ nilpotent. Then every left-invariant generalized
    complex structure on $G_A$ is generalized Calabi--Yau.
\end{cor}
\begin{proof}
      Let $\rho=\exp(\beta)\wedge\Omega$ be a non-degenerate pure spinor defining such a structure. By Lemma  \ref{lemma-integ-equiv_new}, its associated subspace $W$ is $A^T$-invariant. Since $A$ is nilpotent, $A^T|_W$ is nilpotent, yielding $\tr(A^T|_W) = 0$. Theorem \ref{thm-Calabi-Yau} completes the proof. 
\end{proof}

\begin{example} \label{ex:jordan}
    Let $n \ge 2$ and let $G_A$ be an almost abelian Lie group with $A=J_{2n-1}(\lambda)$, a single
    real Jordan block of maximal size, which we take with $1$'s on the superdiagonal, so that
    $A^Te^p=\lambda e^p+e^{p+1}$ with the convention $e^{2n}=0$. By Proposition \ref{prop:k-leq-1},
    the type of any left-invariant generalized complex structure on $G_A$ satisfies $k\leq 1$.

    Type $0$, that is, symplectic type, is known to be attained if and only if $\lambda = 0$
    \cite{Arroyo2025}. We determine exactly when type $1$ is realized.

    Since $\ker\left(A^T-\lambda\right)=\mathrm{span}_{\mathbb C}\{e^{2n-1}\}$ is the unique
    $A^T$-invariant line, a type $1$ pure spinor has, up to scale, $\Omega=z\,e^0+e^{2n-1}$ for some
    $z\in\mathbb C$, and $\Omega\wedge\overline\Omega=(z-\bar z)\,e^0\wedge e^{2n-1}$, so that
    non-degeneracy forces $z\notin\mathbb R$. Neither condition below depends on $z$. First,
    $\omega^{n-1}\wedge\Omega\wedge\overline\Omega\ne0$ amounts, after wedging with
    $e^0 \wedge e^{2n-1}$, to the non-degeneracy of $\omega$ restricted to
    $V:=\mathrm{span}_{\mathbb R}\{e_1,\dots,e_{2n-2}\}$, that is,
    $$ \det[\omega]_V\ne0,$$
    where $[\omega]_V$ denotes the matrix of that restriction; in particular $\omega_{1m}\ne0$ for
    some $2 \le m \le 2n-2$. Second, items (a) and (b) of Lemma \ref{lem:integ_equiv} hold with $X=0$
    and $\xi=-\lambda e^0$, since $\d\Omega=-\lambda e^0\wedge e^{2n-1}=\xi\wedge\Omega$; hence
    integrability reduces to item (c), $\d\beta\wedge\Omega=0$, that is, to
    $$e^{2n-1}\wedge \d\omega=0\qquad\text{and}\qquad e^{2n-1}\wedge \d B=0 .$$
    Writing $\omega_{ab}=\sum_{1 \le i<j \le 2n-1}\omega_{ij}e^i\wedge e^j$ and using
    $\Psi(e^i\wedge e^j)=2\lambda e^i\wedge e^j+e^{i+1}\wedge e^j+e^i\wedge e^{j+1}$, the first
    equation says that the restriction of $\Psi(\omega_{ab})$ to $V$ vanishes; its
    $e^1\wedge e^m$ coefficient yields
    \begin{align} \label{recur}
        2\lambda\,\omega_{1m}+\omega_{1,m-1}=0,\qquad m=2,\dots,2n-2.
    \end{align}
    If $\lambda\ne0$, then, since $\omega_{11}=0$, \eqref{recur} forces $\omega_{1m}=0$ for every
    $m\le 2n-2$ by induction, so that the first row of $[\omega]_V$ vanishes, contradicting
    $\det[\omega]_V\ne0$.

    If $\lambda=0$, then \eqref{recur} only forces $\omega_{1,m-1}=0$ for $m\le2n-2$, leaving
    $\omega_{1,2n-2}$ free. Accordingly, take $\Omega=e^0+\mathbf{i}e^{2n-1}$ and
    $\beta=B+\mathbf{i}\omega$ with $B = 0$ and
    $$\omega=\sum_{j=1}^{n-1}(-1)^{j-1}e^j\wedge e^{2n-1-j},$$
    populating the anti-diagonal of $[\omega]_V$, so $\det[\omega]_V\ne0$. A telescoping computation
    gives $\Psi(\omega)=e^1\wedge e^{2n-1}$, whence $\d\omega=-e^0\wedge e^1\wedge e^{2n-1}$ and
    $$\d\omega\wedge\Omega=\left(-e^0\wedge e^1\wedge e^{2n-1}\right)\wedge
    \left(e^0+\mathbf{i}e^{2n-1}\right)=0,$$
    so that integrability holds. Therefore $\rho = \exp(\mathbf{i}\omega) \wedge \Omega$ defines a
    left-invariant generalized complex structure of type $1$. Moreover $\d\Omega=0$, hence
    $\d\rho = 0$ and the structure is generalized Calabi--Yau, in accordance with Corollary
    \ref{cor:nilpotent-CY}. Therefore type $1$ is also attained on $J_{2n-1}(\lambda)$ if and only if
    $\lambda=0$.

We remark that hypothesis $n\ge2$ is necessary: for $n=1$ and $\lambda\ne0$ one has
    $\mathfrak g_A\cong\mathfrak{aff}(\mathbb R)$, which carries a left-invariant complex structure,
    that is, one of type $n=1$.
\end{example}

\bigskip

\section{Nonexistence results and constructions in dimension 6}
\label{sec:dimension-6}

We study intermediate-type generalized complex structures on
six-dimensional almost abelian Lie groups that admit neither a
left-invariant complex structure nor a left-invariant symplectic
structure. This restriction distinguishes the problem from the
four-dimensional solvable case: by \cite{Barberis}*{Theorem 4.7},
a four-dimensional solvable Lie group admits a left-invariant generalized
complex structure if and only if it admits a left-invariant complex or
symplectic structure. Intermediate type can occur in dimension four,
but not in the absence of both of these classical structures.

Throughout this section, let
\[
\mathfrak g_A=\mathbb R e_0\ltimes_A\mathfrak h,
\qquad \mathfrak h=\mathbb R^5,
\]
and assume that $\mathfrak g_A$ admits neither a complex nor a symplectic
structure. Since types $0$ and $3$ are thereby excluded, the remaining
question is whether types $1$ and $2$ occur. We first identify the Jordan
patterns compatible with this assumption. We then settle the diagonal
case, prove nonexistence for three non-diagonal patterns, and construct
explicit structures for several of the remaining parameter families.
The conditions attached to these constructions are sufficient; a
necessity statement will be made only in the diagonal case.

\subsection{Jordan patterns and conventions}

By Proposition \ref{Conjugate}, we may choose $A$ in real Jordan form.
Write $J_r(t)$ for an upper Jordan block with real eigenvalue $t$ and
ones on the superdiagonal, and set
\[
C(a,b)=\begin{pmatrix}a&-b\\ b&a\end{pmatrix},\qquad b\ne0.
\]
A scalar summand denotes a $1\times1$ block. With respect to the dual
basis $e^0,\ldots,e^5$, a real Jordan chain satisfies
$A^Te^j=t e^j+e^{j+1}$, except at its last covector, where the second
term is absent. We write $e^{i_1\cdots i_r}$ for
$e^{i_1}\wedge\cdots\wedge e^{i_r}$.

Applying the complex and symplectic existence criteria of
\cite{Arroyo2025} leaves the following ten parameterized patterns.
All eigenvalue lists below are \emph{multisets}. Two identical pairs
mean four entries that can be grouped as $\{r,r\}$ and $\{s,s\}$;
two opposite pairs mean four entries grouped as $\{r,-r\}$ and
$\{s,-s\}$. The pairs must use disjoint occurrences, even when some
values coincide; in particular, two zero entries form an opposite pair.

\begin{enumerate}[label=(\arabic*)]
\item $J_5(\lambda)$, with $\lambda\ne0$.
\item $J_4(\lambda)\oplus\mu$, with $\lambda\ne0$.
\item $J_3(\lambda)\oplus J_2(\mu)$, with $\lambda\ne\pm\mu$.
\item $J_3(\lambda)\oplus\mu\oplus\nu$, with
      $\lambda\ne0$ or $\mu\ne-\nu$.
\item $J_2(\lambda)\oplus J_2(\mu)\oplus\nu$, with
      $\lambda\ne\pm\mu$ and
      $\{\lambda,\mu,\nu\}\ne\{0,-\alpha,\alpha\}$
      for every $\alpha\in\mathbb R$.
\item $J_2(\lambda)\oplus\mu\oplus\nu\oplus\gamma$, where
      $\{\lambda,\mu,\nu,\gamma\}$ contains neither two identical
      pairs nor two opposite pairs. If $\lambda=0$, the multiset
      $\{\mu,\nu,\gamma\}$ must also contain no opposite pair.
\item $C(a,b)\oplus J_3(\lambda)$, with $(a,\lambda)\ne(0,0)$.
\item $C(a,b)\oplus J_2(\lambda)\oplus\mu$, with $\lambda\ne\mu$.
      If $a=0$, also require $\lambda\ne-\mu$ and $\lambda\ne0$.
\item $C(a,b)\oplus\lambda\oplus\mu\oplus\nu$, where
      $\{\lambda,\mu,\nu\}$ contains no identical pair.
      If $a=0$, it must also contain no opposite pair.
\item $\operatorname{diag}(\lambda_1,\ldots,\lambda_5)$, whose
      spectrum contains neither two identical pairs nor two opposite pairs.
\end{enumerate}

For completeness, a real $5\times5$ Jordan matrix has either only real
blocks, one nonreal conjugate pair together with three real dimensions,
or four nonreal dimensions together with one real eigenvalue. The
seven partitions of the real dimension five give patterns (1)--(6)
and (10); the three partitions of the remaining real dimension three
give patterns (7)--(9). In the last possibility, the four-dimensional
nonreal part admits a commuting complex structure and is an invariant
hyperplane in $\mathfrak h$, so the corresponding almost abelian Lie
algebra admits a complex structure and is excluded. The restrictions
above remove the complex and symplectic cases within the ten remaining
patterns. These are families of Jordan forms, not a list of ten Lie
algebra isomorphism classes.

\subsection{The diagonal case}

For pattern (10), the general spectral criterion from
Section \ref{sec:gcs-introduction-and-everyting} gives a complete answer.

\begin{prop}\label{prop:diag6}
Suppose that $A$ is diagonal and its spectrum contains neither two
identical pairs nor two opposite pairs. The following are equivalent:
\begin{enumerate}[label=(\alph*)]
\item $G_A$ admits a left-invariant generalized complex structure of type $1$;
\item $G_A$ admits a left-invariant generalized complex structure of type $2$;
\item the spectrum can be reordered as
\[
\operatorname{spec}(A)=\{\lambda,\lambda,\mu,\zeta,-\zeta\}.
\]
\end{enumerate}
\end{prop}

\begin{proof}
Apply Theorem \ref{thm:os-meninos-vem-como} with $n=3$.
For type $1$, its two alternatives are one identical pair and one
opposite pair, or two opposite pairs. The latter is excluded by
hypothesis. For type $2$, the alternatives are two identical pairs,
or one identical pair and one opposite pair. The former is excluded.
Thus both types have exactly the spectral condition in (c).
\end{proof}

\begin{cor}\label{sec3:diag-closed}
Under the hypotheses of Proposition \ref{prop:diag6}, no left-invariant
generalized complex structure of type $1$ or $2$ has a closed invariant
pure-spinor generator.
\end{cor}

\begin{proof}
By Corollary \ref{cor-CY-diag}, closure at type $1$ requires a reordering
as in Proposition \ref{prop:diag6}(c) with $\lambda=0$. This produces
the two opposite pairs $\{0,0\}$ and $\{\zeta,-\zeta\}$.
At type $2$, closure requires such a reordering with $\mu=-\lambda$,
producing the opposite pairs $\{\lambda,\mu\}$ and
$\{\zeta,-\zeta\}$. Both possibilities contradict the hypothesis.
\end{proof}

\subsection{Obstructions for three non-diagonal patterns}

The diagonal criterion depends only on spectral pairings. In the
non-diagonal case, the invariant subspaces and the nilpotent parts of
the Jordan blocks impose additional restrictions. For patterns
(1)--(3), these restrictions rule out both intermediate types.

\begin{prop}\label{prop:nonexistence}
Let $G_A$ have Lie algebra $\mathfrak g_A=\mathbb R e_0\ltimes_A\mathbb R^5$.
If $A$ has one of the following forms, then $G_A$ admits no
left-invariant generalized complex structure:
\begin{enumerate}[label=(\alph*)]
\item $J_5(\lambda)$, with $\lambda\ne0$;
\item $J_4(\lambda)\oplus\mu$, with $\lambda\ne0$;
\item $J_3(\lambda)\oplus J_2(\mu)$, with $\lambda\ne\pm\mu$.
\end{enumerate}
In particular, none admits a left-invariant generalized Calabi--Yau structure.
\end{prop}

\begin{proof}
The parameter restrictions exclude types $0$ and $3$, so it remains
to exclude types $1$ and $2$. Write a hypothetical invariant generator as
\[
\rho=\exp(\beta)\wedge\Omega,\qquad \beta=B+\mathbf i\omega,
\qquad B,\omega\in\bigwedge^2\mathfrak g_A^*.
\]
Let $W\subset\mathfrak h^*\otimes\mathbb C$ be the span of the
restrictions to $\mathfrak h$ of the one-form factors of $\Omega$.
By Lemma \ref{lemma-integ-equiv_new}, $W$ is $A^T$-invariant and has
dimension equal to the type. We use the integrability equation
$\d\beta\wedge\Omega=0$ from Lemma \ref{lem:integ_equiv}, together
with the non-degeneracy condition
$\omega^{3-k}\wedge\Omega\wedge\bar\Omega\ne0$.

\emph{Case (a).} This is the case $n=3$ of Example \ref{ex:jordan}.
Indeed, the unique invariant line is $\mathbb C e^5$, so a type $1$
structure would have $\Omega=ze^0+e^5$. The equation
$\d\omega\wedge e^5=0$ forces the component of $\omega$ in
$\bigwedge^2\operatorname{span}\{e^1,e^2,e^3,e^4\}$ to vanish:
the induced derivation there is $2\lambda I$ plus a nilpotent operator,
and hence is invertible. This contradicts non-degeneracy.
The unique invariant two-plane is
$\operatorname{span}_{\mathbb C}\{e^4,e^5\}$, which is real, so
Lemma \ref{lem:dim-W-conj} excludes type $2$.

\emph{Case (b).} The structure equations are
\begin{equation}\label{eq:MC}
\begin{gathered}
\d e^0=0,\qquad
\d e^1=-\lambda e^{01}-e^{02},\qquad
\d e^2=-\lambda e^{02}-e^{03},\\
\d e^3=-\lambda e^{03}-e^{04},\qquad
\d e^4=-\lambda e^{04},\qquad \d e^5=-\mu e^{05}.
\end{gathered}
\end{equation}
The coefficients needed below follow from
\begin{equation}\label{eq:d2.4.1}
\begin{aligned}
\d e^{12}&=-2\lambda e^{012}-e^{013},\\
\d e^{13}&=-2\lambda e^{013}-e^{023}-e^{014},\\
\d e^{23}&=-2\lambda e^{023}-e^{024}.
\end{aligned}
\end{equation}
Write $\omega=\sum_{i<j}\omega_{ij}e^{ij}$. No other
$\d e^{ij}$ contributes to $e^{012}$, $e^{013}$, or $e^{023}$.
Consequently, for any $\gamma\in\{e^4,e^5,e^{45}\}$, the equation
$\d\omega\wedge\gamma=0$ implies
\begin{equation}\label{type1system}
2\lambda\omega_{12}=0,\qquad
\omega_{12}+2\lambda\omega_{13}=0,\qquad
\omega_{13}+2\lambda\omega_{23}=0.
\end{equation}
Since $\lambda\ne0$, all three coefficients
$\omega_{12},\omega_{13},\omega_{23}$ vanish.

For type $1$, the possibilities for $W$ are $\mathbb C e^4$,
$\mathbb C e^5$, and, when $\lambda=\mu$,
$\mathbb C(e^4+c e^5)$ with $c\in\mathbb C$.
Thus, up to a nonzero scalar, $\Omega=ze^0+u$, where $u$ is one of
these generators. Since every invariant exterior derivative has an
$e^0$ factor, the first two choices give
$\d\omega\wedge e^4=0$ or $\d\omega\wedge e^5=0$.
For the mixed choice, wedge
$\d\beta\wedge(e^4+c e^5)=0$ with $e^5$ before taking imaginary
parts; this gives $\d\omega\wedge e^{45}=0$.
In every case, \eqref{type1system} applies.
Now $\Omega\wedge\bar\Omega$ lies in the span of
$e^{04},e^{05},e^{45}$. The complementary coefficients of
$\omega^2/2$, along $e^{1235},e^{1234},e^{0123}$ respectively, are
\[
\begin{aligned}
&\omega_{12}\omega_{35}-\omega_{13}\omega_{25}+\omega_{15}\omega_{23},\\
&\omega_{12}\omega_{34}-\omega_{13}\omega_{24}+\omega_{14}\omega_{23},\\
&\omega_{01}\omega_{23}-\omega_{02}\omega_{13}+\omega_{03}\omega_{12}.
\end{aligned}
\]
All three vanish, contradicting non-degeneracy.

For type $2$, first suppose $\lambda\ne\mu$. The invariant plane
splits along the generalized eigenspaces and is either
$\operatorname{span}_{\mathbb C}\{e^3,e^4\}$ or
$\operatorname{span}_{\mathbb C}\{e^4,e^5\}$.
Both are real, contrary to Lemma \ref{lem:dim-W-conj}.
If $\lambda=\mu$, put $N=A^T-\lambda I$.
When $N|_W=0$, one has
$W=\ker N=\operatorname{span}_{\mathbb C}\{e^4,e^5\}$, again real.
Otherwise $N|_W$ has one Jordan block of size two.
Since $W\subset\ker N^2=\operatorname{span}_{\mathbb C}\{e^3,e^4,e^5\}$,
its image under $N$ is $\mathbb C e^4$; after choosing a basis of $W$,
\[
W=\operatorname{span}_{\mathbb C}\{e^4,e^3+c e^5\},\qquad
\Omega=(z_1e^0+e^4)\wedge(z_2e^0+e^3+c e^5).
\]
A direct expansion gives
\[
\Omega\wedge\bar\Omega=(z_1-\bar z_1)(\bar c-c)e^{0345}.
\]
Non-degeneracy therefore requires
$\operatorname{Im}z_1\ne0$, $\operatorname{Im}c\ne0$, and
$\omega_{12}\ne0$. However, wedging
$\d\beta\wedge e^4\wedge(e^3+c e^5)=0$ with $e^5$ and taking
imaginary parts gives $e^{345}\wedge\d\omega=0$, whereas
\[
e^{345}\wedge\d\omega=2\lambda\omega_{12}e^{012345}.
\]
This is a contradiction.

\emph{Case (c).} Here the structure equations are
\begin{equation}\label{eq:MC2}
\begin{gathered}
\d e^0=0,\qquad
\d e^1=-\lambda e^{01}-e^{02},\qquad
\d e^2=-\lambda e^{02}-e^{03},\\
\d e^3=-\lambda e^{03},\qquad
\d e^4=-\mu e^{04}-e^{05},\qquad
\d e^5=-\mu e^{05}.
\end{gathered}
\end{equation}
In particular,
\begin{equation}\label{eq:d2.3.2}
\begin{aligned}
\d e^{12}&=-2\lambda e^{012}-e^{013},&
\d e^{13}&=-2\lambda e^{013}-e^{023},\\
\d e^{45}&=-2\mu e^{045},&
\d e^{14}&=-(\lambda+\mu)e^{014}-e^{024}-e^{015},\\
\d e^{24}&=-(\lambda+\mu)e^{024}-e^{034}-e^{025}.&&
\end{aligned}
\end{equation}
For type $1$, the only invariant lines are $\mathbb C e^3$ and
$\mathbb C e^5$, because $\lambda\ne\mu$.
We may therefore write $\Omega=ze^0+e^j$ with $j\in\{3,5\}$;
non-degeneracy requires $\operatorname{Im}z\ne0$.
Integrability implies $\d\omega\wedge e^j=0$.
The coefficients obtained from $e^{014}$ and $e^{024}$ give, for either
choice of $j$,
\begin{equation}\label{systemtwo}
(\lambda+\mu)\omega_{14}=0,\qquad
\omega_{14}+(\lambda+\mu)\omega_{24}=0.
\end{equation}
Thus $\omega_{14}=\omega_{24}=0$, since $\lambda+\mu\ne0$.

If $j=3$, non-degeneracy requires
\begin{equation}\label{omegas}
\omega_{12}\omega_{45}-\omega_{14}\omega_{25}
+\omega_{15}\omega_{24}\ne0,
\end{equation}
so $\omega_{12}\omega_{45}\ne0$.
The coefficients arising from $e^{012}$ and $e^{045}$ in
$\d\omega\wedge e^3=0$ also give
$2\lambda\omega_{12}=0$ and $2\mu\omega_{45}=0$.
Hence $\lambda=\mu=0$, contrary to the hypothesis.
If $j=5$, non-degeneracy instead requires
\begin{equation}\label{omegas2}
\omega_{12}\omega_{34}-\omega_{13}\omega_{24}
+\omega_{14}\omega_{23}\ne0,
\end{equation}
which reduces to $\omega_{12}\omega_{34}\ne0$.
But the coefficients arising from $e^{012}$ and $e^{013}$ give
\[
2\lambda\omega_{12}=0,\qquad
\omega_{12}+2\lambda\omega_{13}=0.
\]
If $\lambda\ne0$, the first equation forces $\omega_{12}=0$;
if $\lambda=0$, the second does. Both contradict non-degeneracy.

Finally, at type $2$, the distinct eigenvalues imply that an invariant
plane is one of
\[
\operatorname{span}_{\mathbb C}\{e^2,e^3\},\qquad
\operatorname{span}_{\mathbb C}\{e^4,e^5\},\qquad
\operatorname{span}_{\mathbb C}\{e^3,e^5\}.
\]
Each is real, so $\dim_{\mathbb C}(W+\bar W)=2<3$, contradicting
Lemma \ref{lem:dim-W-conj}.
\end{proof}

\subsection{Explicit constructions and closed generators}

Having excluded patterns (1)--(3), we now give constructions for
parameter families in patterns (4)--(9). We also include the diagonal
representatives from Proposition \ref{prop:diag6}, so that all the
existence statements in this section can be read together.

Table \ref{tab:spectral} specifies the parameter conditions and the
types constructed. Its case numbers refer to the preceding Jordan
list; the letters in cases 6a--6c and 8a--8b distinguish different
parameter families within the same pattern. Each row is subject to
\emph{both} its displayed conditions and the restrictions in that list.
Whole Jordan blocks may be reordered together with their coframes.
In particular, the two size-two blocks in case (5) may be exchanged;
in case (6), the singleton blocks may be permuted while the size-two
block occupies positions $1,2$.

Table \ref{tab:forms} then gives a representative for each indicated
type, in the coframe ordered by the corresponding matrix in
Table \ref{tab:spectral}. We use
\[
\eta=e^1+\mathbf i e^2,\qquad \tau=a+\mathbf i b,
\qquad A^T\eta=\tau\eta
\]
when a complex block is present. Products of forms in the tables denote
wedge products. The last two columns record the constants used to
verify integrability and non-degeneracy, as follows.

\begin{theo}\label{sec3:constructions}
For each row of Table \ref{tab:forms}, impose the corresponding
conditions of Table \ref{tab:spectral} and set
$\rho=\exp({\mathbf i\omega})\wedge\Omega$.
Then $\rho$ defines a left-invariant generalized complex structure
of the indicated type $k$. More precisely,
\begin{equation}\label{sec3:row-identities}
\d\rho=-\Lambda e^0\wedge\rho,\qquad
\omega^{3-k}\wedge\Omega\wedge\bar\Omega
=c_\rho e^{012345},
\end{equation}
where $\Lambda$ and the nonzero constant $c_\rho$ are listed in the table.
The displayed invariant generator is closed if and only if $\Lambda=0$.
\end{theo}


\begin{table}[htpb]
\centering\small
\renewcommand{\arraystretch}{1.2} 
\caption{Parameter conditions for Theorem \ref{sec3:constructions}. The restrictions in the Jordan list remain in force.}
\label{tab:spectral}
\begin{tabular}{llll}
\toprule
Case & $A$ & Conditions & Types constructed\\
\midrule
4 & $J_3(\lambda)\oplus\mu\oplus\nu$ & $\lambda=0,\ \mu=\nu$ & $1,2$\\
5 & $J_2(\lambda)\oplus J_2(\mu)\oplus\nu$ & $\mu=0,\ \nu=\lambda$ & $1,2$\\
6a & $J_2(\lambda)\oplus\mu\oplus\nu\oplus\gamma$ & $\mu=\lambda,\ \gamma=-\nu$ & $1,2$\\
6b & $J_2(\lambda)\oplus\mu\oplus\nu\oplus\gamma$ & $\mu=\nu=\alpha,\ \gamma=-\lambda$ & $1,2$\\
6c & $J_2(\lambda)\oplus\mu\oplus\nu\oplus\gamma$ & $\lambda=0,\ \mu=\nu$ & $1$\\
7 & $C(a,b)\oplus J_3(\lambda)$ & $b\ne0,\ \lambda=0$ & $1,2$\\
8a & $C(a,b)\oplus J_2(\lambda)\oplus\mu$ & $b\ne0,\ \lambda+\mu=0$ & $1,2$\\
8b & $C(a,b)\oplus J_2(\lambda)\oplus\mu$ & $b\ne0,\ \lambda=0$ & $1,2$\\
9 & $C(a,b)\oplus\lambda\oplus\mu\oplus\nu$ & $b\ne0,\ \lambda+\mu=0$ & $1,2$\\
10 & $\operatorname{diag}(\lambda,\lambda,\mu,\zeta,-\zeta)$ & as displayed & $1,2$\\
\bottomrule
\end{tabular}
\end{table}

\begin{table}[H]
\centering\small 
\renewcommand{\arraystretch}{1.25} 
\caption{Representatives and verification constants for \eqref{sec3:row-identities}. Read each row with the matching conditions in Table \ref{tab:spectral}.}
\label{tab:forms}
\begin{tabular}{llllll}
\toprule
Case & Type & $\Omega$ & $\omega$ & $\Lambda$ & $c_\rho$\\
\midrule
4 & 1 & $e^4+\mathbf i e^5$ & $e^{01}+e^{23}$ & $\mu$ & $-4\mathbf i$\\
4 & 2 & $(\mathbf i e^0+e^3)(e^4+\mathbf i e^5)$ & $e^{12}$ & $\mu$ & $-4$\\
5 & 1 & $e^2+\mathbf i e^5$ & $e^{01}+e^{34}$ & $\lambda$ & $-4\mathbf i$\\
5 & 2 & $(e^0+e^1+\mathbf i e^5)(\mathbf i e^0+e^2)$ & $e^{34}$ & $2\lambda$ & $4$\\
6a & 1 & $e^2+\mathbf i e^3$ & $e^{01}+e^{45}$ & $\lambda$ & $-4\mathbf i$\\
6a & 2 & $(e^0+e^1+\mathbf i e^3)(\mathbf i e^0+e^2)$ & $e^{45}$ & $2\lambda$ & $4$\\
6b & 1 & $e^3+\mathbf i e^4$ & $e^{01}+e^{25}$ & $\alpha$ & $-4\mathbf i$\\
6b & 2 & $(\mathbf i e^0+e^2)(e^3+\mathbf i e^4)$ & $e^{15}$ & $\lambda+\alpha$ & $4$\\
6c & 1 & $e^3+\mathbf i e^4$ & $e^{05}+e^{12}$ & $\mu$ & $-4\mathbf i$\\
7 & 1 & $e^0+\eta$ & $e^{12}+e^{03}+e^{45}$ & $\tau$ & $-4\mathbf i$\\
7 & 2 & $\eta(e^0-\mathbf i e^5)$ & $e^{34}$ & $\tau$ & $-4$\\
8a & 1 & $e^0+\eta$ & $e^{12}+e^{03}+e^{45}$ & $\tau$ & $-4\mathbf i$\\
8a & 2 & $\eta(e^0-\mathbf i e^4)$ & $e^{35}$ & $\tau+\lambda$ & $4$\\
8b & 1 & $e^0+\eta$ & $e^{12}+e^{34}+e^{05}$ & $\tau$ & $-4\mathbf i$\\
8b & 2 & $\eta(e^0-\mathbf i e^5)$ & $e^{34}$ & $\tau+\mu$ & $-4$\\
9 & 1 & $e^0+\eta$ & $e^{34}+e^{05}$ & $\tau$ & $-4\mathbf i$\\
9 & 2 & $\eta(e^0-\mathbf i e^5)$ & $e^{34}$ & $\tau+\nu$ & $-4$\\
10 & 1 & $e^1+\mathbf i e^2$ & $e^{03}+e^{45}$ & $\lambda$ & $-4\mathbf i$\\
10 & 2 & $(e^1+\mathbf i e^2)(e^0+\mathbf i e^3)$ & $e^{45}$ & $\lambda+\mu$ & $4$\\
\bottomrule
\end{tabular}
\end{table}

\begin{proof}
Extend $A^T$ to the degree-zero derivation $\Psi$ of
$\bigwedge^*\mathfrak g_A^*\otimes\mathbb C$ by setting $\Psi(e^0)=0$.
Thus
\[
\Psi(\alpha\wedge\delta)=\Psi(\alpha)\wedge\delta
+\alpha\wedge\Psi(\delta),\qquad
\d\alpha=-e^0\wedge\Psi(\alpha).
\]
This reduces the verification to the action on the one-form factors
and on the summands of $\omega$.

First, applying the block formulas to each $\Omega$ gives
\begin{equation}\label{sec3:dOmega}
\d\Omega=-\Lambda e^0\wedge\Omega
\end{equation}
with the listed value of $\Lambda$. One can compute $\Lambda$ directly
as the trace of $A^T$ on the span of the restrictions of the factors
of $\Omega$ to $\mathfrak h$. For example, in case 5 at type $2$,
these restrictions are $u=e^1+\mathbf i e^5$ and $v=e^2$;
they satisfy $A^Tu=\lambda u+v$ and $A^Tv=\lambda v$.
Consequently,
$\Psi(u\wedge v)=2\lambda u\wedge v$.
The terms in $\Omega$ containing $e^0$ have zero exterior derivative,
which proves \eqref{sec3:dOmega} for that row. The same calculation
applies to case 6a at type $2$, with $u=e^1+\mathbf i e^3$.
In the other rows, the restricted factors are eigenvectors, so their
eigenvalues simply add to give $\Lambda$.

Next, all derivatives of the displayed two-forms are given by
\[
\d\omega=
\begin{cases}
-e^{013}, & \text{case 4, type }2,\\
-e^{025}, & \text{case 6b, type }2,\\
-e^{035}, & \text{case 7, type }2,\\
-e^{045}, & \text{case 8a, type }2,\\
-2a e^{012}, & \text{cases 7, 8a, and 8b, type }1,\\
0, & \text{all other rows.}
\end{cases}
\]
Each of these forms wedges to zero with its corresponding $\Omega$.
In the three complex-block rows of type $1$, this follows from
$e^{012}\wedge(e^0+\eta)=0$; in the four exceptional type $2$ rows,
it follows from the common one-form factors after the $e^0$ terms
are discarded. This proves $\d\omega\wedge\Omega=0$ in every row.
For instance, in case 6b the condition $\gamma=-\lambda$ gives
\[
\Psi(e^{25})=(\lambda+\gamma)e^{25}=0,\qquad
\Psi(e^{15})=(\lambda+\gamma)e^{15}+e^{25}=e^{25}.
\]
Thus the type $1$ two-form is closed, whereas the type $2$ two-form
has derivative $-e^{025}$, which vanishes upon wedging with
$(\mathbf i e^0+e^2)\wedge(e^3+\mathbf i e^4)$.
This illustrates why integrability requires
$\d\omega\wedge\Omega=0$, without requiring $\d\omega=0$.

Combining these identities gives
\[
\begin{aligned}
\d\rho
&=\exp({\mathbf i\omega})\wedge
  \bigl(\mathbf i\,\d\omega\wedge\Omega+\d\Omega\bigr)\\
&=-\Lambda e^0\wedge\rho.
\end{aligned}
\]
The right-hand side is the Clifford action of the complex one-form
$-\Lambda e^0$, so the integrability criterion applies.
The forms $\Omega$ are decomposable of the indicated degree, and
expanding $\omega^{3-k}\wedge\Omega\wedge\bar\Omega$ yields the
nonzero constants $c_\rho$ listed in Table \ref{tab:forms}.
For example, in case 6b at type $2$,
\[
\Omega\wedge\bar\Omega=-4e^{0234},\qquad
\omega\wedge\Omega\wedge\bar\Omega
=-4e^{15}\wedge e^{0234}=4e^{012345}.
\]
Thus every displayed spinor is non-degenerate and has the stated type.
Finally, $e^0\wedge\Omega\ne0$ in every row; since multiplication by
$\exp({\mathbf i\omega})$ is invertible, also $e^0\wedge\rho\ne0$.
Hence \eqref{sec3:row-identities} implies that $\d\rho=0$ exactly
when $\Lambda=0$.
\end{proof}

The values of $\Lambda$ also identify which of these
representatives yield invariant generalized Calabi--Yau structures.
Here the restriction excluding the extremal types is essential: some
formal solutions of $\Lambda=0$ lie outside the parameter range of
this section.

\begin{cor}\label{sec3:closed-representatives}
Under the standing hypotheses, the only closed invariant generators
among the representatives in Table \ref{tab:forms} occur in case 6b,
at type $2$, with
\[
\alpha=-\lambda,\qquad \lambda\ne0.
\]
Their matrices and generators are
\[
\begin{gathered}
A=J_2(\lambda)\oplus(-\lambda)\oplus(-\lambda)\oplus(-\lambda),\\
\rho=\exp({\mathbf i e^{15}})\wedge
(\mathbf i e^0+e^2)\wedge(e^3+\mathbf i e^4).
\end{gathered}
\]
These Lie groups admit no lattice.
\end{cor}

\begin{proof}
For the complex-block rows, $\operatorname{Im}\Lambda=b\ne0$,
so the generators are not closed. The diagonal rows are covered by
Corollary \ref{sec3:diag-closed}.
In case 4, $\Lambda=\mu$, and $\mu=0$ together with
$\lambda=0$ and $\mu=\nu$ violates the restriction for pattern (4).
In case 5, $\Lambda=0$ forces $\lambda=0$; since $\mu=0$,
this violates $\lambda\ne\pm\mu$.
In case 6a, it again forces $\lambda=0$, so $\mu=0$ and
$\gamma=-\nu$, contrary to the additional restriction for pattern (6)
when $\lambda=0$. In case 6c, it forces $\mu=\nu=0$;
these two singleton entries form an opposite pair, again excluded
when $\lambda=0$.

In case 6b at type $1$, closure would require $\alpha=0$.
The multiset $\{\lambda,\mu,\nu,\gamma\}$ would then be
$\{\lambda,0,0,-\lambda\}$, containing two opposite pairs.
At type $2$, however, closure requires $\alpha=-\lambda$.
For $\lambda\ne0$, the resulting multiset
$\{\lambda,-\lambda,-\lambda,-\lambda\}$ contains neither two
identical pairs nor two opposite pairs, so the standing restrictions
are satisfied. The value $\lambda=0$ is excluded by those restrictions.
Finally,
\[
\operatorname{tr}A=2\lambda-3\lambda=-\lambda\ne0.
\]
The group is therefore non-unimodular and cannot admit a lattice.
\end{proof}

Proposition \ref{prop:nonexistence} settles patterns (1)--(3), and
Proposition \ref{prop:diag6} settles the diagonal pattern (10).
For patterns (4)--(9), Theorem \ref{sec3:constructions} provides
explicit sufficient conditions. In particular, listing only type $1$
in case 6c does not assert nonexistence of type $2$, and a parameter
choice absent from Table \ref{tab:spectral} is not excluded by the
construction argument. Likewise, Corollary
\ref{sec3:closed-representatives} concerns the displayed generators;
it does not classify all invariant generalized Calabi--Yau structures
on the non-diagonal families.

\bibliographystyle{amsplain} 
 
\bibliography{main}

\medskip
 
\appendix 
\section{Algebraic bounds on invariant subspaces}
\label{appendix:A}

The proof of Theorem \ref{thm:type-obstruction} asks that $\mathcal{W}_\lambda$ maximizes $\dim_{\mathbb{C}} W - \dim_{\mathbb{C}}(W \cap \overline{W})$ among all invariant subspaces. In this appendix, we prove this by analyzing the algebraic properties of nilpotent operators on complexified vector spaces. 
Throughout this section, let $V$ be a finite-dimensional real vector space and $N: V \to V$ a nilpotent operator, extended $\mathbb{C}$-linearly to $V_{\mathbb{C}}$.  
 For any $N$-invariant subspace $W \subset V_{\mathbb{C}}$, we set
  $$c(W) := \dim_{\mathbb{C}} W - \dim_{\mathbb{C}}(W \cap \overline{W}).$$
We recall that the \emph{Jordan type} of $N$ on $V$ is the partition $\mu=(\mu_1\ge \mu_2\ge\cdots)$ of its block sizes. A \emph{Jordan chain} of length $m$ associated to $N$ is a sequence of linearly independent vectors $\{w_j\}_{j=1}^{m}$ such that $Nw_j = w_{j+1}$ for $1\leq j < m$ and $Nw_m = 0$. 

\begin{lemma}\label{lem:2}
Let $W \subset V_\mathbb{C}$ be an $N$-invariant subspace. Define $R := W \cap \overline{W}$ and $U := W + \overline{W}$. Then$$U/R = Z \oplus \overline{Z},$$ where $ Z := W/R$ and $\dim_\mathbb{C} Z = c(W)$.
\end{lemma}

\begin{proof}
    Note that $R$ and $U$ are invariant under complex conjugation. The canonical projection onto $U/R$ induces natural injections $W/R \hookrightarrow U/R$ and $\overline{W}/R \hookrightarrow U/R$.  These images span $U/R$ and intersect trivially. Indeed, if $x \in W$ and $y \in \overline{W}$ have the same image in $U/R$, then $x - y \in R \subset \overline{W}$, implying $x \in W \cap \overline{W} = R$. Defining $Z := W/R$ and  $\overline{Z} := \overline{W}/R$, we obtain $U/R = Z \oplus \overline{Z} $. 
\end{proof}

\begin{lemma}\label{lem:3}
Let $R\subseteq U\subseteq V$ be $N$-invariant, let $\nu_1 \geq \nu_2 \geq \dots \geq 0 $ be the Jordan type of $N$ on $U/R$, and $\mu_1 \geq \mu_2 \geq \dots \geq 0 $ the Jordan type of $N$ on $V$. Then $\nu_i\le\mu_i$ for all $i$.
\end{lemma}
\begin{proof}
Let $\tau$ be the Jordan type on an invariant subspace $U\subseteq V$.
For every $r\ge1$, counting Jordan blocks gives
\[
\#\{j:\tau_j\ge r\}
=\dim(\ker N\cap N^{r-1}U)
\le\dim(\ker N\cap N^{r-1}V)
=\#\{j:\mu_j\ge r\}.
\]
If $\tau_i>\mu_i$, choosing $r=\tau_i$ makes the left count
at least $i$ and the right count at most $i-1$, a contradiction.
Thus $\tau_i\le\mu_i$.
For an invariant $R\subseteq U$, the dual of $U/R$ is the invariant
annihilator of $R$ in $U^*$. Transposition preserves Jordan type,
so applying the subspace result to this annihilator gives
$\nu_i\le\tau_i\le\mu_i$. The argument works over $\mathbb R$
and over $\mathbb C$.
\end{proof}

\begin{lemma}\label{lem:c-bound}
Let $\mu = (\mu_1 \ge \dots \ge \mu_s)$ be the Jordan type of $N$ on $V_\mathbb{C}$. For any $N$-invariant subspace $W \subset V_\mathbb{C}$,
\[
c(W) \leq \sum_{j=1}^{\lfloor s/2 \rfloor} \mu_{2j}.
\]
\end{lemma}

\begin{proof}
    With $R, U, Z$ as in Lemma \ref{lem:2}, let $\tau=(\tau_1\ge\cdots\ge \tau_r)$ be the Jordan type of $N$ acting $\mathbb{C}$-linearly on $Z$. Since $R = \overline{R}$ by definition, complex conjugation induces a well-defined antilinear bijection $\sigma: Z \to \overline{Z}$ given by $\sigma([w]) = [\overline{w}]$. We also note that $N$ naturally commutes with complex conjugation, in particular, $N\sigma([w]) = \sigma(N[w])$.  

For the $j$-th Jordan block of $N$ in $Z$, fix a basis $\{w_1, \dots, w_{\tau_j}\}$ satisfying $Nw_k = w_{k+1}$ for $k < \tau_j$ and $Nw_{\tau_j} = 0$. 

Applying $\sigma$ to our chosen basis yields the conjugated image $\{\overline{w_1}, \dots, \overline{w_{\tau_j}}\}$, and $ N(\overline{w_k}) = \overline{N(w_k)} = \overline{w_{k+1}}$ 
for $k < \tau_j$, while $N(\overline{w_{\tau_j}}) = 0$. Thus, the conjugate vectors form a Jordan chain of the same length $\tau_j$ in $\overline{Z}$. Consequently, the operator $N$ on $\overline{Z}$ has Jordan type $\tau$, as it does on $Z$.

By Lemma \ref{lem:2}, $U/R = Z \oplus \overline{Z}$. Since both summands are $N$-invariant, the union of their respective Jordan bases yields a linearly independent Jordan basis for the quotient $U/R$. Thus, the Jordan type of $N$ on $U/R$ is the termwise merge of the partition $\tau$ with itself, resulting in:
$$\nu = (\tau_1 \ge \tau_1 \ge \tau_2 \ge \tau_2 \ge \dots \ge \tau_r \ge \tau_r).$$
We now apply Lemma \ref{lem:3} to the $N$-invariant complex subspaces $R \subseteq U \subseteq V_{\mathbb{C}}$. Once the Jordan type of $N$ on $V_{\mathbb{C}}$ is identical to its Jordan type $\mu$ on $V$, we obtain: 
$$\nu_i \le \mu_i, \quad\forall i.$$
Reading it in the
even-indexed terms, $\nu_{2j}=\tau_j$, so
$$\tau_j \le \mu_{2j}, \qquad 1\le j\le r.$$
 Recalling that $c(W)=\dim_{\mathbb{C}} Z =\sum_{j=1}^r \tau_j$, 
$$c(W) = \sum_{j=1}^r \tau_j \le \sum_{j=1}^r \mu_{2j} \le \sum_{j=1}^{\lfloor s/2\rfloor}\mu_{2j},$$
Here $2r\le s$, since $\nu_{2r}=\tau_r>0$ and $\nu_{2r}\le\mu_{2r}$. The last inequality adds nonnegative terms.
\end{proof}

\end{document}